\documentclass[11pt]{amsart}
\usepackage[numeric]{amsrefs}
\DeclareRobustCommand{\VAN}[3]{#2}
\usepackage[utf8]{inputenc}
\usepackage{tikz-cd}
\usepackage[top=4cm, bottom=3cm,left=2.5cm,right=2.5cm]{geometry}

\usepackage[english]{babel}
\newcommand{\twopartdef}[4]
{
	\left\{
		\begin{array}{ll}
			#1 & \mbox{if } #2 \\
			#3 & \mbox{if } #4
		\end{array}
	\right.
}
\usepackage{amsmath}
\usepackage{mathtools} 
\usepackage{amsfonts}
\usepackage{amssymb} 
\usepackage{stmaryrd}
\usepackage{amsthm}
\usepackage{thmtools}
\usepackage{thm-restate}
\newtheorem{Def}[subsubsection]{Definition}
\newtheorem{Thm}[subsubsection]{Theorem}
\newtheorem{Lem}[subsubsection]{Lemma}
\newtheorem{Prop}[subsubsection]{Proposition}
\newtheorem{Cor}[subsubsection]{Corollary}
\newtheorem{Conj}{Conjecture}
\newtheorem*{Thm*}{Theorem}

\newtheorem{Lemm}[subsection]{Theorem}

\theoremstyle{remark}

\newtheorem{Rem}[subsubsection]{Remark}
\newtheorem{Claim}[subsubsection]{Claim}
\newtheorem{Assump}[subsubsection]{Assumption}

\newcommand{\gl}{\operatorname{GL}}
\newcommand{\galq}{\gal(\overline{\mathbb{Q}}/\mathbb{Q})}
\newcommand{\frob}{\operatorname{Frob}}
\newcommand{\sym}{\operatorname{Sym}}
\newcommand{\coker}{\operatorname{coker}}

\newcommand{\zz}{\mathbb{Z}}
\newcommand{\rr}{\mathbb{R}}
\newcommand{\cc}{\mathbb{C}}
\newcommand{\qqq}{\mathbb{Q}}
 \newcommand{\TJM}[4]{
\bigl(\begin{smallmatrix}
  #1 & #2 \\
  #3 & #4
\end{smallmatrix} \bigr)}
 
\newcommand{\apsp}{A_{K,\overline{s}}}
\newcommand{\apgen}{A_{K,\overline{\eta}}}
\newcommand{\ahsp}{A_{H,\overline{s}}}
\newcommand{\ahgen}{A_{H,\overline{\eta}}}
\newcommand{\aksp}{A_{Q,\overline{s}}}
\newcommand{\akgen}{A_{Q,\overline{\eta}}}
\newcommand{\para}{K}
\newcommand{\gsp}{\operatorname{GSp}}
\newcommand{\im}{\operatorname{Im}}
\newcommand{\gal}{\operatorname{Gal}}
\newcommand\restr[2]{{
  \left.\kern-\nulldelimiterspace 
  #1 
  \vphantom{\big|} 
  \right|_{#2} 
  }}

\mathtoolsset{showonlyrefs}
\usepackage{bbm}
\usepackage{graphicx}
\usepackage{enumitem}
\usepackage{hyperref}
\begin{document}

\title[Geometric Jacquet-Langlands for paramodular Siegel threefolds]{A geometric Jacquet-Langlands correspondence for paramodular Siegel threefolds}
\author[van Hoften]{Pol van Hoften}
\email{pol.van\_hoften@kcl.ac.uk}
\address{King's College London}
\keywords{Shimura varieties, Siegel modular forms, Weight-monodromy }
\thanks{This work was supported by the Engineering and Physical Sciences Research Council [EP/L015234/1], The EPSRC Centre for Doctoral Training in Geometry and Number Theory (The London School of Geometry and Number Theory), University College London and King's College London.}
\begin{abstract}
We study the Picard-Lefschetz formula for Siegel modular threefolds of paramodular level and prove the weight-monodromy conjecture for its middle degree inner cohomology. We give some applications to the Langlands programme: Using Rapoport-Zink uniformisation of the supersingular locus of the special fiber, we construct a geometric Jacquet-Langlands correspondence between $\gsp_4$ and a definite inner form, proving a conjecture of Ibukiyama \cite{IbukiyamaConjecture}. We also prove an integral version of the weight-monodromy conjecture and use it to deduce a level lowering result for cohomological cuspidal automorphic representations of $\gsp_4$.
\end{abstract} 
\maketitle

\section{Introduction}
In this paper we will study the cohomology of paramodular Siegel threefolds, using geometric results of \cite{YuParamodular} and the Picard-Lefschetz formula of \cite{SGA7}. Our first main result (Theorem \ref{JLtheorem}) is a geometric Jacquet-Langlands correspondence for such paramodular Siegel threefolds, which proves a conjecture of Ibukiyama \cite{IbukiyamaConjecture} that we will explain below. Our second main result (Theorem \ref{mazursprinciple}) is a level lowering result for cohomological cuspidal automorphic representations of $\gsp_4$, in the spirit of Mazur's level lowering results for modular forms. Our main technical result (Theorem \ref{WMtheorem}) is the weight-monodromy conjecture for the inner cohomology of Siegel threefolds over $\mathbb{Q}_p$ with paramodular level structure at $p$.
\subsection{Ibukiyama's conjecture}
Let $D$ be the quaternion algebra over $\qqq$ that is non-split over $\rr$ and $\qqq_p$ for a prime $p$ and split at all other places. Let $S_k[\Gamma_0(p)]^{\text{new}}$ be the space of modular forms of weight $k \ge 2$ and level $\Gamma_0(p)$ that are $p$-new. A classical result of Eichler gives a Hecke equivariant injection (surjective for $k \ge 3$)
\begin{align}
    S_k[\Gamma_0(p)]^{\text{new}} \xhookrightarrow{} \mathcal{A}^{D^{\times}}_k[\mathcal{O}_D^{\times}], \label{new}
\end{align}
where $\mathcal{A}^{D^{\times}}_k[\mathcal{O}_D^{\times}]$ denotes the space of algebraic modular forms for $D^{\times}$ of weight $k$ and level $\mathcal{O}_D^{\times}$. These can be described explicitly as `functions' on the (finite) class set of $D$ which are easier to understand computationally than modular forms. In fact this description is used in practice to compute bases of newforms for spaces of modular forms. This correspondence is a special case of the Jacquet-Langlands correspondence between the algebraic group $\gl_2/\qqq$ and its inner form $D^{\times}/\qqq$, proven in \cite{JacquetLanglands} using the trace formula.

The quaternionic similitude group $G=GU_2(D)$ associated with $D$ is an inner form of $\gsp_4$, such that $G(\rr)$ is compact modulo centre. General conjectures of Langlands predict a transfer from automorphic representations of $G$ to automorphic representations of $\gsp_4$. A particular instance of this transfer was first conjectured by Ihara and Ibukiyama \cites{Ihara, Ibukiyama1} and later extended by Ibukiyama. We will state a slight reformulation of Conjecture 5.1 of \cite{IbukiyamaConjecture} below. Let $S_{k,j}[\para(p)]$ denote the space of Siegel modular forms of weight $k,j$ and level $\para(p)$, where $\para(p)$ is the paramodular group (c.f. Section \ref{cohomology}). Let $\mathcal{A}^G_{k,j}[K_2(p)]$ denote the space of algebraic modular forms for $G$ of weight $k,j$ and level $K_2(p)$, where $K_2(p)$ is an analogue of the paramodular group (c.f. Section \ref{Sec:AlgModForms}).
\begin{Conj}[Ibukiyama] \label{Conj1}
For $k \ge 0$ and for $j \ge 3$ there is an injective map
\begin{align}
    S_{k,j}[\para(p)]^{\text{new}} \xhookrightarrow{} \mathcal{A}^G_{k,j}[K_2(p)], \label{newer}
\end{align}
which is Hecke-equivariant for the prime-to-$p$ Hecke operators. 
\end{Conj}
\begin{Rem}
This conjecture can be used to do explicit computations with Siegel modular forms, we give some examples: In \cite{Dembele} the author constructs an algorithm computing Hecke eigenvalues of Siegel modular forms, assuming a form of Conjecture \ref{Conj1}; in \cite{LanskyPollack} the authors do extensive computations with algebraic modular forms for $G$; in \cite{Fretwell} the author uses the conjecture to find computational evidence for Harder's conjecture on congruences between Siegel modular forms and elliptic modular forms.
\end{Rem}
It is an old idea of Serre that the classical Jacquet-Langlands correspondence between modular forms and quaternion modular forms can be realised geometrically by restricting mod $p$ modular forms (which are sections of a line bundle on the modular curve $Y_0(N)$) to the supersingular locus. This mod $p$ Jacquet-Langlands correspondence can be upgraded to a characteristic zero correspondence by considering the Picard-Lefschetz formula in \'etale cohomology \cite{SGA7} for the modular curve $Y_0(Np)$. Ghitza \cite{Ghitza} generalised Serre's ideas to produce a mod $p$ Jacquet-Langlands correspondence for Siegel modular forms (by restricting to the superspecial locus). Our approach to Conjecture \ref{Conj1} is then to use the Picard-Lefschetz formula for the Siegel threefold with paramodular level at $p$, which is possible because Yu \cite{YuParamodular} computed its singularities. Using results of \cite{KudlaRapoport} and \cite{YuSuperSingular} we can identify the finite set of singular points of $X \otimes \mathbb{F}_p$ with an adelic double quotient of the form
\begin{align}
    G(\qqq) \setminus G(\mathbb{A}_f)/(U^p K_2(p)).
\end{align}
Algebraic modular forms for $G$ are basically functions on this finite set and the Picard-Lefschetz formula gives us a map from the space of these algebraic modular forms to the middle cohomology of $X_{\overline{\qqq}_p}$. This map has an interpretation in terms of the action of the inertia group and our first result concerns this action.
\begin{restatable}{mainThm}{WMtheorem} \label{WMtheorem}
Let $X/\qqq_p$ be the Siegel threefold of neat level $U=U^p K(p)$ and let $\mathbb{V}$ be an automorphic local system of $L$-vector spaces where $L/\qqq_\ell$ is a finite extension ($\ell \not=p)$ of sufficiently regular weight. Then the weight-monodromy conjecture holds for $H^3_{!}(X_{\overline{\qqq}_p}, \mathbb{V})$. \end{restatable}
\begin{Rem}
The result, without restrictions on the weight or level, also follows from Arthur's classification of automorphic forms on $\gsp_4$ together with results on weight-monodromy for $\operatorname{GL}_4$ by Caraiani \cites{Caraiani}. The point of the theorem is to give a geometric proof, using the cohomological vanishing theorems of \cites{LanSuhII}. In fact we will need an integral refinement of the theorem for our level-lowering result, which does not follow from automorphic considerations.
\end{Rem}
\begin{Rem}
When $X$ is a non-compact Shimura variety of Hodge type, with minimal compactification $j:X \xhookrightarrow{} X^{\ast}$, one expects the weight-monodromy conjecture to hold for $H^{\bullet}(X^{\ast}, j_{! \ast} \mathbb{V})$.  Indeed, it follows from the usual weight-monodromy conjecture for smooth projective varieties, see Appendix \ref{Appendix:A}. In general the inner cohomology $H_{!}^{\bullet}(X,\mathbb{V})$ is only a sub-quotient of $H^{\bullet}(X^{\ast}, j_{! \ast} \mathbb{V})$ and it is unclear to the author if one should expect that the weight-monodromy conjecture holds for it. However when $X$ is a Siegel threefold, the inner cohomology $H_{!}^{\bullet}(X,\mathbb{V})$ is a direct summand of $H^{\bullet}(X^{\ast}, j_{! \ast} \mathbb{V})$, because it is equal to the cuspidal cohomology.
\end{Rem}
At this point we have all the ingredients to prove Conjecture \ref{Conj1}, but we will actually prove more. We can work with arbitrary level away from $p$ to transfer certain cuspidal automorphic representations $\pi$ of $\gsp_4$ to $G$, which is a geometric incarnation of the Jacquet-Langlands correspondence between $G$ and $\gsp_4$. We will state a short version of our main theorem below, see Theorem \ref{JLtheorem2} for a more general statement.
\begin{restatable}{mainThm}{JLtheorem} \label{JLtheorem} \begin{enumerate}[label=(\arabic*)]
\item Let $\pi$ be a non-endoscopic cohomological cuspidal automorphic representation of $\gsp_4$ such that $\pi_{\infty}$ is in the discrete series and such that $\pi_p$ is ramified and $K(p)$-spherical. Then there is a cuspidal automorphic representation $\sigma$ of $G$ such that $\pi_v \cong \sigma_v$ for finite places $v \not=p$, such that $\sigma_p$ is $K_2(p)$-spherical and with $\sigma_{\infty}$ determined by $\pi_{\infty}$. Moreover, $\sigma$ occurs with multiplicity one in the cuspidal spectrum of $G$.

\item Let $k \ge 0, j \ge 3$ and let $N$ be a squarefree integer with $p \mid N$, then there is an injective map
\begin{align}
\varphi:S_{k,j}[\para(N)]^{p-\text{new}} \xhookrightarrow{} \mathcal{A}^G_{k,j}[K_2(N)]
\end{align}
equivariant for the prime-to-$p$ Hecke operators, which proves Conjecture \ref{Conj1}. \end{enumerate}
\end{restatable}
\begin{Rem}
Sorensen constructs a Jacquet-Langlands transfer in \cite{SorensenRaising} using the stable trace formula. However he has to assume that there is a certain auxiliary prime $q \not=p$ where $\sigma_q$ is ramified enough, which is never satisfied for paramodular forms of squarefree level. Moreover, while I was writing up this paper, Weissauer and R\"osner informed me that they can construct a general Jacquet-Langlands correspondence between $\gsp_4$ and its inner forms, also using trace formula methods, see their forthcoming article \cite{WeissauerRoesner}.
\end{Rem}
\begin{Rem}
Ibukiyama gives a conjectural characterisation of the image of $\varphi$ when $N=p$. We will give a different characterisation of its image in Theorem \ref{JLtheorem2}, which is probably easier to verify in practice. We will also discuss the image of the map $\pi \mapsto \sigma$ in Theorem \ref{JLtheorem2}.
\end{Rem}
\subsection{Mazur's principle}
In this section we describe a level lowering result for cohomological cuspidal automorphic representations of $\gsp_4$. We start by recalling a classical result of Mazur concerning levels of modular forms, which answers the following question: Given a normalised eigenform $f \in S_2[\Gamma_0(Np)]$ that is $p$-new, is there a congruence $f \equiv g \mod \ell$, for some $\ell \not=p$, with $g \in S_2[\Gamma_0(N)]$? We can translate this congruence into an isomorphism of mod $\ell$ Galois representations $\overline{\rho}_{f,\ell} \cong \overline{\rho}_{g,\ell}$. Since the Galois representation $\rho_{f,\ell}$ is unramified at $p$, a necessary condition for such a congruence to exist is that the Galois representation $\overline{\rho}_{f,\ell}$ is unramified at $p$, and Serre conjectured in \cite{Serre} that the converse should be true. Ribet proved this conjecture in \cite{Ribet}, which was famously used to show that modularity of semi-stable elliptic curves implies Fermat's last theorem. Below we state a slightly weaker version of Ribet's theorem, due to Mazur:
\begin{Thm*}(Mazur's principle, Theorem 6.1 in \cite{Ribet}) \label{classicalMazur}
Assume that $\overline{\rho_{f,\ell}}$ is irreducible, unramified at $p$ and that it has two distinct $\frob_p$ eigenvalues (equivalently, $p \not=1 \mod \ell$). Then there is a normalised eigenform $g \in S_2[\Gamma_0(N)]$ with $\overline{\rho_{g,\ell}} \simeq \overline{\rho_{f,\ell}}$.
\end{Thm*}
Our second main result is the following analogue of Mazur's principle for cohomological cuspidal automorphic representations of $\gsp_4$. 
\begin{restatable}{mainThm}{mazursprinciple} 
\label{mazursprinciple}
Let $\pi$ be a cuspidal automorphic representation of $\gsp_4$ that is cohomological of weight $a > b > 0$ such that $\pi_p$ is ramified and $K(p)$-spherical. Let $U=U^p \cdot K(p) \subset \gsp_4(\mathbb{A}_f)$ be a neat compact open subgroup such that $\pi^U \not=0$. Let $\ell \not=p$ be a prime such that $\overline{\rho_{\pi, \ell}}$ is irreducible, such that the group $U_{\ell}$ is hyperspecial at $\ell$ and such that $\ell>a+b+4$.
Then if $\overline{\rho_{\pi, \ell}}$ is unramified at $p$ and has four distinct $\frob_p$ eigenvalues, there exists a cuspidal automorphic representation $\pi'$, of the same weight and level $U^p$ away from $p$, such that $\overline{\rho_{\pi, \ell}}=\overline{\rho_{\pi',\ell}}$ and such that $\pi'_p$ is unramified.
\end{restatable}
\begin{Rem}
We need the assumption that $a>b>0$, that $\ell>a+b+4$ and that $K_{\ell}$ is hyperspecial to apply Theorem 10.1 of \cite{LanSuhII}, which is a torsion vanishing result. Our proof only needs this torsion vanishing result after localising at a `nice' maximal ideal of the Hecke algebra, so it should be possible to relax these assumptions. After a first version of this paper appeared on the arxiv, Haining Wang proved a version of Proposition \ref{componentgroupiseisenstein} localised at such a maximal ideal (see Proposition 5.3 of \cite{Wang}). This means that the assumption that $a>b>0$ and that $K_{\ell}$ hyperspecial is hyperspecial is no longer necessary. The assumption that $\overline{\rho_{\pi, \ell}}$ is irreducible and has four distinct Frobenius eigenvalues is crucial for the actual level lowering argument. We note that Haining Wang proves a version of Theorem \ref{mazursprinciple} under different assumptions, see Theorem 5.1 of \cite{Wang}. To be precise he has weaker assumptions on the Frobenius eigenvalues, but has to assume semi-simplicity of the middle degree \'etale cohomology of the Siegel threefold.
\end{Rem}
An essential ingredient in the proof of Mazur's principle is the fact that the component group of the Jacobian of the modular curve $X_0(Np)$ is Eisenstein. Following Jarvis and Rajaei \cites{Jarvis, Rajaei} we define a cohomological analogue of the component group and show that it is zero if the aforementioned torsion-vanishing result holds. This statement can be seen as an integral refinement of the weight-monodromy conjecture because the component group is defined as the cokernel of a certain monodromy operator.

\begin{Rem}
Gee and Geraghty \cite{GeeGeraghty} prove more general level lowering results for cuspidal automorphic representations of $\gsp_4$, under the assumption that $\overline{\rho_{\pi, \ell}}$ has large image and is `ordinary', see Theorem 7.5.2 of op. cit. for a precise statement. It might be possible to generalise their results using recent work of Yamauchi \cites{Yamauchi}. Moreover Sorensen proves a potential level lowering result for $\gsp_4$ in \cite{SorensenPotential}.
\end{Rem}
\begin{Rem}
The paramodular group $\para(Np)$ is not neat, so a level lowering result from $S_{k,j}[\para(Np)]$ to $S_{k,j}[\para(N)]$ does not follow. However, it should be possible to prove such a result by cleverly choosing some auxiliary level structure (but we would need some extra conditions on $\overline{\rho_{\pi, \ell}}$, c.f. the main theorem of \cite{Jarvis}).
\end{Rem}

\subsection{Overview of the paper}
We start by giving an overview of the theory of Siegel modular forms and automorphic representations of $\gsp_4(\mathbb{A})$ in Section \ref{automorphic}. In Section \ref{cohomology} we discuss the Vogan-Zuckerman classification and Arthur's classification and use these to describe the cohomology of Siegel threefolds. We define algebraic modular forms for $G$ in Section \ref{Sec:AlgModForms}. In Section \ref{integral} we survey some results on integral models of Siegel threefolds and describe the combinatorics of their supersingular loci explicitly in terms of Shimura sets for $G$. Section \ref{Eisenstein} is the technical heart of the paper where we prove the vanishing of the component group. We deduce the main theorems from this in Section \ref{applications}.
\section{Automorphic forms} \label{automorphic}
In this section we will discuss (cuspidal) automorphic representations of $\gsp_4(\mathbb{A})$. We will first describe classical Siegel modular forms of vector-valued weight for $\gsp_4$, which give rise to such automorphic representations. We then carry out some local computations needed in our proof of Mazur's principle. We end by recalling the Galois representations associated with cohomological cuspidal automorphic representations of $\gsp_4(\mathbb{A})$. 

\subsection{Siegel modular forms of genus \texorpdfstring{$2$}{2}}
We define the group $\gsp_4$ as the group scheme over $\mathbb{Z}$ defined by the functor sending a commutative ring $R$ to
\begin{align}
    \left\{(g, \lambda) \in \operatorname{GL}_4(R) \times R^{\times} \; | \; J = \lambda g J g^t \right\},
\end{align}
where $J=\TJM{0}{1_2}{-1_2}{0}$ and $1_2$ is the $2 \times 2$ identity matrix. The projection map $(g, \lambda) \mapsto \lambda$ is called the similitude character and is denoted by
\begin{align}
    \operatorname{sim}:\gsp_4 \to \mathbb{G}_m.
\end{align} The group $\gsp_4(\rr)$ acts transitively on the Siegel upper half space $\mathbb{H}_2^{\pm}$ defined as
\begin{align}
     \{z=x+iy \in \operatorname{Mat}_{2 \times 2}(\mathbb{C}) \; | \; z \text{ symmetric and } y \text{ positive- or negative definite} \}
\end{align}
by $\TJM{a}{b}{c}{d} \cdot z:=(az+b)(cz+d)^{-1}$. Let $\Gamma \subset \gsp_4(\qqq)$ be a congruence subgroup and for $k \ge 0, j \ge 0$ consider the irreducible representation of $\gl_2(\cc)$ given by 
\begin{align}
    W_{k,j}:=\sym^k V \otimes \operatorname{det}^j V,
\end{align}
where $V$ is the standard representation of $\gl_2(\cc)$.
\begin{Def}
A Siegel modular form of weight $k,j$ and level $\Gamma$ is a holomorphic function $f: \mathbb{H}_2^{\pm} \to W_{k,j}$ such that such that for all $z \in \mathbb{H}_2^{\pm}$ and all $\TJM{a}{b}{c}{d} \in \Gamma$ we have
\begin{align}
f \left(\frac{az+b}{cz+d} \right)&=\rho_{k,j}(cz+d) f(z).
\end{align}
A Siegel modular form is called a cusp-form if $\Phi(f)=0$, where $\Phi(f)$ is the Siegel operator (c.f. \cite[pp. 12]{vdGeer}), we write $S_{k,j}[\Gamma]$ for the $\cc$-vector space of cusp forms of weight $k,j$ and level $\Gamma$.
\end{Def}
We will be particularly interested in the case that $\Gamma$ is given by the paramodular group $\para(N)$ for some squarefree $N$. Here 
\begin{align}
    \para(N)=\left \{ g \in \gsp_4(\qqq) \cap 
    \left(\begin{smallmatrix}
    \mathbb{Z} &  N \zz & \zz &  \zz \\
    \zz & \zz & \zz &  N^{-1} \zz \\
     \zz & N \zz & \zz &  \zz \\
    N \zz & N \zz & N \zz &  \mathbb{Z}
    \end{smallmatrix} \right) \right \}.
\end{align}
We usually like to think of our levels as compact open subgroups in $\gsp_4(\mathbb{A}_f)$ when translating Siegel modular forms into automorphic representations. For a prime $p$ we define
\begin{align}
    H=H(p)&:=\gsp_4(\mathbb{Z}_p) \\ 
    K=K(p)&:= \left \{ g \in \gsp_4(\mathbb{Q}_p) \cap 
    \left(\begin{smallmatrix}
    \mathbb{Z}_p & p \mathbb{Z}_p & \mathbb{Z}_p &  \mathbb{Z}_p \\
    \mathbb{Z}_p & \mathbb{Z}_p & \mathbb{Z}_p &  p^{-1} \mathbb{Z}_p \\
     \mathbb{Z}_p & p\mathbb{Z}_p & \mathbb{Z}_p &  \mathbb{Z}_p \\
    p \mathbb{Z}_p & p \mathbb{Z}_p & p \mathbb{Z}_p &  \mathbb{Z}_p
    \end{smallmatrix} \right)
    \; : \operatorname{sim}(g) \in \mathbb{Z}_p^{\times} \right \} \\
    Q=Q(p)&:= K(p) \cap H(p),
\end{align}
we call $K$ the paramodular group and $Q$ the Klingen parahoric. The paramodular group $K(N)$ now corresponds to the compact open subgroup
\begin{align}
    \prod_{p \mid N} K(p) \prod_{p \nmid N} \gsp_4(\zz_p) \subset \gsp_4(\mathbb{A}_f).
\end{align}
We will implicitly use the following result throughout the paper (c.f. \cite[Theorem 2]{AsgariSchmidt}):
\begin{Prop} \label{SiegeltoAutomorphic}
Given a normalised eigenform $f \in S_{k,j}(\para(N))$ with $N$ squarefree then there is an associated cuspidal automorphic representation $\pi$ of $\gsp_4(\mathbb{A})$ such that $\pi_{\infty}$ is in the holomorphic discrete series and such that $\pi_{f}$ is unramified for $p \nmid N$ and $\pi_{f}^{K(p)} \not=0$ for $p \mid N$.
\end{Prop}

\subsection{The local Langlands correspondence} \label{local}
This section is about irreducible smooth representations $\pi$ of $\gsp_4(\qqq_p)$, that are ramified and $K(p)$-spherical, and their associated Weil-Deligne representations under the local Langlands correspondence of \cite{locallanglands}. Table A.13 of \cite{LocalNewforms} describes all of these representations and their associated Weil-Deligne representations can be found in Chapter 2.4 of op. cit. These representations break up into $5$ types (IIa, IVc, Vb, Vc, VIc) of which only one (type IIa) is generic (c.f. Table A1 of op. cit.). 

Let $v$ be the normalised absolute value $v:\qqq_p^{\times} \to \overline{\qqq_{\ell}}$, let $\sigma$ be an unramified character of $\qqq_p^{\times}$ and let $\chi$ be a character of $\qqq_p^{\times}$ such that $\chi^2 \not= v^{\pm 1}$ and $\chi \not= v^{\pm 3/2}$. We consider all of these as character of the Weil group via local class field theory (uniformisers go to geometric Frobenius elements). The representation $\chi \operatorname{St}_{\gl_2} \rtimes \sigma$ of type IIa is generic and its associated Weil-Deligne representation is given by:
\begin{align}
\left( \begin{smallmatrix}
\chi^2 \sigma & 0 & 0 & 0 \\
0 & v^{1/2} \chi \sigma &  0 & 0 \\
0 & 0 & v^{-1/2} \chi \sigma & 0 \\
0 & 0 & 0 & \sigma 
\end{smallmatrix} \right), \qquad 
N:=\left(\begin{smallmatrix}
0 & 0 & 0 & 0 \\
0 & 0 & 1 & 0 \\
0 & 0 & 0 & 0 \\
0 & 0 & 0 & 0 
\end{smallmatrix} \right).
\end{align}
We note that the central character of this representation is $\chi^2 \sigma^2$ which is also equal to the similitude character of the Weil-Deligne representation. Furthermore, the Weil-Deligne representation satisfies the weight-monodromy conjecture, i.e., it is pure in the sense of Taylor-Yoshida \cite[pp.6]{TaylorYoshida}. More concretely, this means that the monodromy operator has rank `as large as possible'.

\begin{Rem} \label{genericremark}
The other four types (IVc, Vb, Vc, VIc) are not generic, and their associated Weil-Deligne representations do not satisfy the weight-monodromy conjecture. After we prove Theorem \ref{WMtheorem} (weight-monodromy), we can use this observation to deduce that certain representations must be of type IIa, which will be used in the proof of Theorem \ref{mazursprinciple}.
\end{Rem}
\subsubsection{Atkin-Lehner eigenvalues} \label{Atkin-Lehner}
The element
\begin{align}
    u:=\left(\begin{smallmatrix}
    0 & 0 & 1 & 0 \\
    0 & 0 & 0 & -1 \\
    p & 0 & 0 & 0 \\
    0 & -p & 0 & 0
    \end{smallmatrix} \right)
\end{align}
normalises the paramodular group $K$ and so it acts on the space of invariants $\pi^K$. Moreover since $u^2=p I$ is central it must act on $\pi^K$ via the central character of $\pi$.
\begin{Lem} \label{eigenvaluelemma}
Let $\pi=\chi \operatorname{St}_{\gl_2} \rtimes \sigma$ be a representation of type IIa with $\chi, \sigma$ unramified and central character $\chi^2 \sigma^2$. Then $\pi^K$ is one-dimensional and $u$ acts via the scalar $(\chi \sigma)(p)$
\end{Lem}
\begin{proof}
When $\chi \sigma$ is trivial, Roberts and Schmidt compute that
$u$ acts by (c.f. Table A.12 in \cite{LocalNewforms})
\begin{align}
    (\chi \sigma)(p)=1,
\end{align}
one can reduce to this case by twisting and using the fact that $\chi, \sigma$ are unramified.
\end{proof}
\subsection{Galois representations}
Galois representations associated with cohomological cuspidal automorphic representations of $\gsp_4(\mathbb{A})$ were constructed in \cite{Taylor}, \cite{Laumon}, \cite{Weissauer}. They satisfy local-global compatibility with the local Langlands correspondence of \cite{locallanglands} by work of Mok \cite[Theorem 3.5]{Mok}. To summarise we state the following version of Theorem 3.5 of \cite{Mok}
\begin{Thm}
Let $\pi$ be a cohomological cuspidal automorphic representation of $\gsp_4(\mathbb{A})$ that is not CAP. Then for every prime number $\ell$ there exists a continuous Galois representation
\begin{align}
    \rho_{\pi, \ell}:\galq \to \gsp_4(\overline{\mathbb{Q}_{\ell}})
\end{align}
such that:
\begin{enumerate}
    \item The representation is unramified at primes $p \not= \ell$ where $\pi_p$ is unramified, de Rham at $\ell$ and crystalline at $\ell$ if $\pi_{\ell}$ is unramified.
    
    \item For all primes $p \not= \ell$ we have local-global compatibility with the local Langlands correspondence of \cite{locallanglands}. To be precise we have
    \begin{align}
        \operatorname{WD} \left( \restr{\rho_{\pi, \ell}}{G_{\mathbb{Q}_p}} \right)^{\text{F-ss}} \cong \operatorname{rec}_p(\pi_p \otimes \operatorname{sim}(v^{3/2})),
    \end{align}
    where $\operatorname{rec}_p$ is the Langlands reciprocity map of \cite{locallanglands} and $\text{F-ss}$ denotes Frobenius semi-simplification of a Weil-Deligne representation. 
    \item The representation satisfies
    \begin{align}
        \rho_{\pi, \ell}^{\vee} \cong \rho_{\pi, \ell} \otimes \operatorname{sim}(\rho_{\pi, \ell})
    \end{align}
    where $\operatorname{sim}(\rho_{\pi, \ell})$ is the similitude character of $\rho_{\pi, \ell}$.
\end{enumerate}
\end{Thm}
\begin{Rem} \label{MokRegular}
If $\pi$ is CAP then the above theorem also holds except that local global compatibility is only true up to semi-simplification (this is Proposition 3.4 of \cite{Mok}).
\end{Rem}

\section{Cohomology of Siegel threefolds} \label{cohomology}
In this section we will describe the inner cohomology $H^3_{!}$ of Siegel threefolds in terms of cuspidal automorphic representations of $\gsp_4$ using the Vogan-Zuckerman classification (c.f. \cites{Taylor, Petersen}). We then make use of Arthur's classification \cites{Arthur, GeeTaibi} of these cuspidal automorphic representations to make our discussion explicit.
\subsection{Automorphic local systems} \label{automorphiclocalsystems}
Let $Y_U$ be the Shimura variety for $(\gsp_4, \mathbb{H}_2^{\pm})$ of level $U=U^{\ell} U_{\ell} \subset \gsp_4(\mathbb{A}_f)$. Irreducible representations $V$ of $\gsp_4(\cc)$ give rise to local systems $\mathbb{V}$ on $Y_U$, called automorphic local systems, as in \cite[pp.24]{Thorne}, \cite[section 3]{LanStroh}): 

Recall that irreducible representations $V$ of $\gsp_4(\cc)$ are parametrized by tuples $(a \ge b \ge 0, c)$, where $a \ge b \ge 0$ determines the restriction of $V$ to $\operatorname{Sp}_4$ and $c$ determines the action of the similitude character. If $\ell>a+b+4$ then the lattice $V_{\mathcal{O}}$ is unique up to homothety because $V_{\mathcal{O}}/\varpi$ is an irreducible representation of $H$, which makes the local system $\mathbb{V}_{\mathcal{O}}$ independent on the choice of lattice. Changing $c$ amounts to Tate-twisting the local system, which has no interesting effect on the cohomology. Therefore we will set $c=a+b$, which normalises things such that $\mathbb{V}_{1,0}$ is the local system corresponding to the (dual of the) relative $\ell$-adic Tate module of the universal abelian variety. In general the local system $\mathbb{V}_{a,b} \otimes L$ occurs as a direct summand of
\begin{align}
\mathbb{V}_{1,0}^{\otimes (a+b)},
\end{align}
and these local systems are self-dual up to a Tate twist. 

\subsection{Cohomology theories} \label{Matsushima}
Before studying the Galois representations that occur in the \'etale cohomology of $Y_U$, we will first study the Hodge structure on the singular cohomology. Write $\mathbb{V}_{\cc}$ for the $\cc$-linear local system on $Y_U(\mathbb{C})$ of weight $a \ge b \ge 0$. One is usually interested in the intersection cohomology
\begin{align}
    H^{\bullet}(\overline{Y}_U(\mathbb{C}), j_{!\ast} \mathbb{V}_{\cc}),
\end{align}
where $j:Y_U \to \overline{Y}_U$ is the inclusion of $Y_U$ into its Baily-Borel compactification $\overline{Y}_U$. By the Zucker conjecture (proven in \cite{Looijenga} and \cite{SaperStern}) this is equal to the $L^2$ cohomology
\begin{align}
    H^{\bullet}_{(2)}(Y_U(\mathbb{C}), \mathbb{V}_{\cc}),
\end{align}
which can be described in terms of automorphic representations as follows: 
\begin{align}
    H^{\bullet}_{(2)}(Y_U(\mathbb{C}), \mathbb{V}_{\cc}) = \bigoplus_{\pi} m(\pi) \pi_{\text{fin}}^U \otimes H^{\bullet}(\mathfrak{g}, K_{\infty} ; \mathbb{V}_{\cc} \otimes \pi_{\infty}),
\end{align}
where the sum runs over automorphic representations $\pi = \pi_{\infty} \otimes \pi_{\text{fin}}$ occurring in the discrete spectrum of $L^2(\gsp_4(\qqq) \setminus \gsp_4(\mathbb{A}))$ with multiplicity $m(\pi)$. There is a natural direct summand $H^{\bullet}_{\text{cusp}}(Y_U(\mathbb{C}), \mathbb{V}_{\cc})$ of cuspidal cohomology which is the direct sum over cuspidal automorphic representations $\pi$. Moreover it follows from the results announced in \cite{Arthur} and proven in \cite{GeeTaibi} that $m(\pi) \le 1$ for all $\pi$.
In this article we will work with the inner cohomology which is defined to be the image of the natural map
\begin{align}
    H^{\bullet}_{!}(Y_U(\cc), \mathbb{V}_{\cc}):= \operatorname{Im} \bigg(H^{\bullet}_c(Y_U(\cc), \mathbb{V}_{\cc}) \to H^{\bullet}(Y_U(\cc), \mathbb{V}_{\cc}) \bigg)
\end{align}
where subscript $c$ denotes compactly supported cohomology. In middle degree there is an isomorphism (c.f. \cite[pp. 294]{Taylor})
\begin{align}
    H^3_{!}(Y_U(\cc), \mathbb{V}_{\cc}) \cong H^3_{\text{cusp}}(Y_U(\cc), \mathbb{V}_{\cc}).
\end{align}
\subsection{Vogan-Zuckerman classification}
The Vogan-Zuckerman classification (c.f. \cite{VoganZuckerman}) describes the possible $\pi_{\infty}$ for which
\begin{align}
    H^{i}(\mathfrak{g}, K_{\infty} ; \mathbb{V}_{\cc} \otimes \pi_{\infty})
\end{align}
is nonzero, where $\mathbb{V}_{\cc}$ is a local system of weight $a \ge b \ge 0$. If $\pi_{\infty}$ is in the discrete series, then it is always cohomological. It turns out that the discrete series representations are the only ones that contribute to middle degree cuspidal cohomology. For a fixed choice of weight $a,b$, there are precisely two choices of $\pi_{\infty}$ in the discrete series, which we will denote by $\pi^H$ and $\pi^W$, where the $H$ stands for `holomorphic' and the $W$ for `Whittaker' (the latter is generic). The $(\mathfrak{g}, K_{\infty})$-cohomology of these representations is two-dimensional and occurs in degree $3$ with Hodge numbers $(3,0),(0,3)$ and $(2,1),(1,2)$ respectively. We note that the representations $\pi$ with $\pi_{\infty}=\pi^H$ correspond to holomorphic cuspidal Siegel modular forms of weight $a-b, b+3$ (c.f. Proposition \ref{SiegeltoAutomorphic}).
\subsection{Arthur's classification}
Arthur has classified the cuspidal automorphic representations $\pi$ in \cite{Arthur} into six different types. Only three of these contribute to the inner cohomology in middle degree:
\begin{itemize}
    \item The cuspidal automorphic representations $\pi$ of general type, which occur with multiplicity one. These are stable which means that both $\pi_{\text{fin}} \otimes \pi^H$ and $\pi_{\text{fin}} \otimes \pi^W$ are cuspidal automorphic. This implies that the four-dimensional Galois representation associated with $\pi$ occurs in the inner cohomology of the Siegel threefold.
    
    \item The cuspidal automorphic representations $\pi$ of Yoshida type which occur with multiplicity $m(\pi)=1$. These are unstable which means that precisely one of $\pi_{\text{fin}} \otimes \pi^H$ and $\pi_{\text{fin}} \otimes \pi^W$ is cuspidal automorphic. These are related to pairs $(\pi_1, \pi_2)$ of cuspidal automorphic representations of $\gl_2$, which we will make precise in Section \ref{Sec:Yoshida}.
    
    \item The cuspidal automorphic representations $\pi$ of Saito-Kurokawa type which are CAP and occur with multiplicity one in the cuspidal spectrum. They only occur in scalar weight (so $a=b$) and have $\pi_{\infty}$ isomorphic to either $\pi^H, \pi^W, \pi^{+}, \pi^{-}$ (the latter two don't contribute to middle degree cohomology). These lifts come from cuspidal automorphic representations $\sigma$ of $\gl_2$ of weight $a+b+4$ and the associated Galois representation is
    \begin{align}
        \rho_{\pi, \ell}=\rho_{\sigma,\ell} \oplus \mathbb{Q}_{\ell}(-b-2) \oplus \mathbb{Q}_{\ell}(-b-1).
    \end{align}
    This Galois representation cannot occur in the middle degree cohomology by purity (outside a finite set of places the inner cohomology will be pure of weight $a+b+4$). In fact we will only see the two-dimensional Galois representation $\rho_{\sigma,\ell}$ in the middle cohomology.
\end{itemize}
\subsection{Yoshida lifts} \label{Sec:Yoshida}
We will now describe Yoshida lifts in terms of cuspidal automorphic representations of $\gl_2$. Let $\pi_1, \pi_2$ be cuspidal automorphic representations of $\gl_2(\mathbb{A})$ with the same central character and suppose that $\pi_{1, \infty}, \pi_{2,\infty}$ are in the holomorphic discrete series with Harish-Chandra parameters $r_1-1, r_2-2$ satisfying $r_1 > r_2 \ge 2$ with $r_1=a+b+4$ and $r_2=a-b+2$.
\begin{Def}
Let $\pi$ be a cuspidal automorphic representation of $\gsp_4(\mathbb{A})$ that is not CAP and such that $\pi_{\infty}$ is in the discrete series. We say that $\pi$ is a Yoshida lift of $(\pi_1, \pi_2)$ if for almost all places $v$ the Satake parameter $\pi_v$ agrees with the Satake parameter of $(\pi_1, \pi_2)$ via the inclusion of dual groups $\gl_2 \times_{\mathbb{G}_m} \gl_2 \to \gsp_4$.
\end{Def}
Alternatively, we can explicitly construct an $L$-packet $L(\pi_1, \pi_2)$ as the restricted tensor product of local $L$-packets $L(\pi_{1,v}, \pi_{2,v})$ for all places $v$. These local $L$-packets are constructed using the $\theta$-correspondence between $\gsp_4$ and $(\gl_2 \times \gl_2/\mathbb{G}_m)$ and have size two if $v$ is infinite or if $v$ is finite and both $\pi_{1,v}$ and $\pi_{2,v}$ are essentially square-integrable. In all other cases the local $L$-packet has size one and moreover always contains a generic member. However not all members of this $L$-packet are cuspidal, which is quantified in the following result of Weissauer.
\begin{Thm}(c.f. Theorem 5.2 of \cite{WeissauerEndoscopy}) \label{multiplicityformula}
A cuspidal automorphic representation $\pi$ is a Yoshida lift of $(\pi_1, \pi_2)$ if and only if it is a cuspidal member of the $L$-packet $L(\pi_1, \pi_2)$. Moreover, elements $\pi$ of this $L$-packet occur in the cuspidal spectrum with multiplicity
\begin{align}
    m(\pi)=\frac{1+(-1)^{(c(\pi))}}{2}
\end{align}
where $c(\pi)$ is the number of places where $\pi_v$ is non-generic. Moreover $\pi_{\infty}$ is cohomological with respect to the local system of weight $a,b$.
\end{Thm}
The following result describes what Galois representation occur in the inner cohomology:
\begin{Prop}(c.f. Corollary 4.2 of \cite{WeissauerEndoscopy})
Let $\pi$ be a Yoshida lift of $(\pi_1, \pi_2)$, then $\pi$ is in the discrete series which means that the $(\mathfrak{g}, K_{\infty})$-cohomology of $\pi_{\infty}$ is two-dimensional. The Galois action on this two-dimensional piece is given by
\begin{align}
    \twopartdef{\rho_{\pi_{1},\ell}}{\pi_{\infty}=\pi^H}{\rho_{\pi_{2,\ell}}(-b-1)}{\pi_{\infty}=\pi^W}.
\end{align}
\end{Prop}
\begin{Def} \label{def:type1}
We call $\pi$ an irrelevant Yoshida lift if $\pi_p$ is ramified but the associated two-dimensional Galois representation occurring in cohomology is unramified. 
\end{Def}
\begin{Cor} \label{Yoshida}
If $\pi \in L(\pi_1, \pi_2)$ has nonzero invariant vectors for $\para(N)$ for squarefree $N$, then $L(\pi_1, \pi_2)$ has precisely one cuspidal member $\pi$, with $\pi_v$ generic for all places $v$. In particular there are no holomorphic Yoshida lifts of level $\para(N)$ and so the Galois representation that we see in the cohomology is
\begin{align}
    \rho_{\pi_{2,\ell}}(-b-1).
\end{align}
\end{Cor}
\begin{proof}
This is just Remark 3.5 in Section 3.3 of \cite{SahaSchmidt}.
\end{proof}
\subsection{Vanishing theorems}
In this section we discuss vanishing theorems for
\begin{align}
H^{\bullet}_c(Y_U, \mathbb{V}_{a,b}),
\end{align}
where we will now take $\mathbb{V}_{a,b}$ to be a $\mathcal{O}$-linear local system, note that these are only well defined if $\ell>a+b+4$. 
\begin{Thm}(Theorem 10.1, Corollary 10.2 of \cite{LanSuhII}) \label{vanishingA}
Let $Y_U$ be the Siegel threefold of neat level $U$ and let $\mathbb{V}=\mathbb{V}_{a,b}$ be an automorphic local system of $\mathcal{O}$-modules. Assume that $a > b > 0$, that $U_{\ell}$ is hyperspecial and that $\ell>a+b+4$. Then 
\begin{align}
    H^3_{!}(Y_{U,\overline{\qqq}}, \mathbb{V}_{\mathcal{O}})
\end{align}
is a free $\mathcal{O}$-module of finite rank and the natural map (with $\mathbb{F}=\mathcal{O}/\varpi$)
\begin{align}
    H^3_{!}(Y_{U,\overline{\qqq}}, \mathbb{V}_{\mathcal{O}}) \otimes_{\mathcal{O}} \mathbb{F} \to H^3_{!}(Y_{U,\overline{\qqq}}, \mathbb{V}_{\mathcal{O}} \otimes_{\mathcal{O}} \mathbb{F})
\end{align}
is surjective. Moreover 
\begin{align}
    H^i_{c}(Y_{U,\overline{\qqq}}, \mathbb{V}_{\mathcal{O}}))&=0 \text{ for } i>3 \\
    H^i(Y_{U,\overline{\qqq}}, \mathbb{V}_{\mathcal{O}}))&=0 \text{ for } i<3.
\end{align}
\end{Thm}
Choosing $\ell$ sufficiently large and using comparison results between singular cohomology and \'etale cohomology we get results for local systems with coefficients in $L$, without restrictions on $\ell$.
\begin{Cor} \label{rationalvanishing}
If we now use $\mathbb{V}_{a,b}$ to denote an $L$-linear local system of weight $a>b>0$ then we have
\begin{align}
    H^i_{c}(Y_{U, \overline{\qqq}}, \mathbb{V}_{a,b})&=0 \text{ for } i>3 \\
    H^i(Y_{U, \overline{\qqq}}, \mathbb{V}_{a,b})&=0 \text{ for } i<3.
\end{align}
\end{Cor}
\begin{Rem}
Theorem 10.1 of op. cit. has the assumptions that $\mu=(a,b,c) \in X_{\gsp_{4}}^{++,<\ell}$, that $|\mu|_{\text{re},+}<\ell$ and that $|\mu|'_{\text{comp}} \le \ell-2$. Working through the definitions we find the first condition is equivalent to $a>b>0$ and $\ell \ge a+b+5$, the second condition is equivalent to $\ell \ge a+b+2$ and the third condition is equivalent to $\ell \ge a+b+4$. We conclude that the conditions $a>b>0$ and $\ell>a+b+4$ imply all the hypotheses of Theorem 10.1 of \cite{LanSuhII} in the case of Siegel modular threefolds.
\end{Rem}
\section{Algebraic modular forms} \label{Sec:AlgModForms}
In this section we will discuss algebraic modular forms on the group $G=GU_2(D)$. The general reference for this section is \cite{Gross}, we have also drawn from chapter 3 of \cite{FretwellThesis}.

Let $G/\qqq$ be a connected reductive group such that $G(\rr)$ is connected and compact modulo centre, in this paper we will use $G=D^{\times}$ and $G=GU_2(D)$ with $D/\qqq$ a definite quaternion algebra. Let $L/\qqq_{\ell}$ be a finite extension as before and let $V$ be an irreducible algebraic $L$-linear representation of $G(\qqq_{\ell})$. Let $U=U^{\ell}U_{\ell}$ be a compact open subgroup of $G(\mathbb{A}_f)$, then we define the $L$-vector space of algebraic modular forms of level $U$ with coefficients in $V$
\begin{align}
    A^{G}_{V}[U]:=\left\{ F: G(\qqq) \setminus G(\mathbb{A}_f)/U^{\ell} \to V \; | \; F(g \cdot u_{\ell}) = u_{\ell}^{-1} F(g) \text{ for } u_{\ell} \in U_{\ell} \right \}.
\end{align}
This vector space is finite dimensional because the double coset spaces
\begin{align}
G(\qqq) \setminus G(\mathbb{A}_f)/U
\end{align}
are finite (Proposition 4.3 of \cite{Gross}). The following result is well known and can be proven as in \cite[Section 2]{Guerberoff} or as in \cite[Section 2]{DiamondTaylor}:
\begin{Prop}
Choose an embedding $L \xhookrightarrow{} \mathbb{C}$, then there is a Hecke-equivariant isomorphism
\begin{align}
    A^{G}_{V}[U] \otimes_{E} \mathbb{C} \cong \bigoplus_{\sigma} m(\sigma) \sigma^U
\end{align}
where $\sigma$ runs over cuspidal automorphic representations of $G$ such that $\sigma_{\infty} \cong V \otimes_{E} \mathbb{C}$ and $m(\sigma)$ is the multiplicity with which $\sigma$ occurs in the discrete spectrum of $G$.
\end{Prop}
\subsection{The inner form}
Let $D$ be the  definite quaternion algebra over $\qqq$ ramified at a single prime $p$ with canonical involution $d \mapsto \overline{d}$ and consider the algebraic group $G/\qqq$ whose $R$ points are given by
\begin{align}
    G(R):=\left\{ (g, \lambda) \in \gl_2(D \otimes_{\qqq} R) \times R^{\times} \; | \; ^t\overline{g}g = \lambda I \right\},
\end{align}
where $\overline{g}$ is entry-wise application of $d \mapsto \overline{d}$ and $I$ is the identity matrix. One can show that $G(\rr)/Z(G(\rr)) \cong \operatorname{USp}(4)/ \{ \pm I \}$ where $\operatorname{USp}(4)$ is the compact $\rr$-form of $\operatorname{Sp}_4$ and that for primes $q \not=p$ there are isomorphisms $G(\qqq_q) \cong \gsp_4(\qqq_q)$. Both of these results are worked out in detail in Section 3.5 of \cite{FretwellThesis}. The statement of Conjecture \ref{Conj1} concerns algebraic modular forms of weight $k,j$ and level $K_2(p)$ and we will discuss these now.
\subsection{Weights}
For $\ell \not=p$ choose an isomorphism $G_{\mathbb{Q}_{\ell}} \simeq \gsp_{4, \mathbb{Q}_{\ell}}$. For every irreducible representation $V$ of $\gsp_4$ we get an irreducible representation $V$ of $G$ defined over $\mathbb{Q}_{\ell}$, as before these are parametrised parametrized by integers $(a \ge b \ge 0, c)$, just like the local systems on Siegel threefolds. Siegel modular forms of weight $k,j$ correspond to the local system of weight $k+j-3, j-3, k+2j-6$ and we define
\begin{align}
    \mathcal{A}^G_{k,j}[U]:=\mathcal{A}^G_{V}[U],
\end{align}
where $V$ denotes the irreducible representation of highest weight $(k+j-3,j-3, k+2j-6)$. We will later see that cohomology of $\mathbb{V}_{a,b}$ with support in the supersingular locus of the Siegel modular threefold can be identified with a space of algebraic modular forms for $G$ of weight $a-b,b+3$.
\subsection{Levels} \label{localII}
Our level subgroups $U \subset G(\mathbb{A}_f)$ will be of the form $U^p U_p$. The level $U^p$ will come from $\gsp_4$ under the isomorphism $G(\mathbb{A}_f^{p}) \cong \gsp_4(\mathbb{A}_f^{p})$, the level $U_p$ will be a maximal compact subgroup of $G(\qqq_p)$ which we will describe now. Let $B$ be the non-split quaternion algebra over $\mathbb{Q}_p$ with standard involution $b \mapsto \overline{b}$ and uniformiser $\varpi_B$, then $G(\qqq_p)$ can also be described as (because all Hermitian forms on $B^{\oplus 2}$ are equivalent)
\begin{align}
    G(\qqq_p) = \left \{(g, \lambda) \in \gl_2(B) \times \qqq_p^{\times} : \overline{g} \begin{pmatrix} 0 & 1 \\ 1 & 0 \end{pmatrix} g = \lambda \begin{pmatrix} 0 & 1 \\ 1 & 0 \end{pmatrix} \right \}.
\end{align}
We will be interested in the maximal compact subgroup 
\begin{align}
    K_2&=\left \{(g, \lambda) \in  G(\qqq_p) : g \in \left( \begin{smallmatrix} \mathcal{O}_B & \varpi_B \mathcal{O}_B \\ \varpi_B^{-1} \mathcal{O}_B & \mathcal{O}_B \end{smallmatrix} \right), \; \lambda \in \zz_p^{\times} \right \},
\end{align}
which is the stabiliser $\mathcal{O}_B \oplus \varpi_B \mathcal{O}_B$. Now let $N$ be a squarefree integer with $p \mid N$, then we define
\begin{align}
    K_2(N):=K_2(p) \times \prod_{v \mid \tfrac{N}{p}} K(v) \prod_{v \nmid N} \gsp_4(\zz_v).
\end{align}
\subsection{Theta lifts}
In section 3 of \cite{IbukiyamaParamodular} Ibukiyama defines a notion of '$\theta$-lift' and suggests that these are precisely the algebraic modular forms that do \emph{not} correspond to holomorphic Siegel modular forms. He predicts (Conjecture 5.3 of \cite{IbukiyamaConjecture}) that the $\theta$-lifting gives injective maps to $\mathcal{A}^G_{k,j}[K_2(p)]$ from
\begin{align} \label{lifts}
\twopartdef{S_{2j-2+k}[\Gamma_0(1)] \times S_{k+2}[\Gamma_0(p)]^{\text{new}}}{k>0 \text{ or } j \text{ odd}}{S_{2j-2}[\Gamma_0(1)] \times M_{2}[\Gamma_0(p)] }{k=0 \text{ and } j \text{ even}}.
\end{align}
We remark that $\operatorname{dim} M_{2}[\Gamma_0(p)]=1+\operatorname{dim} S_{2}[\Gamma_0(p)]$ so that Ibukiyama's formula essentially predicts a compact version of the Saito-Kurokawa lifting
\begin{align}
S_{2j-2}[\Gamma_0(1)] \to \mathcal{A}^G_{0,j}[K_2(N)]
\end{align}
for odd $j$. We will actually prove that there are injective lifts from \eqref{lifts} to cuspidal automorphic representations of $G$, but it is unclear to us that these are $\theta$-lifts in the sense of Ibukiyama (which is why we do not prove Conjecture 5.1 of \cite{IbukiyamaConjecture} but something slightly different). It will follow from our arguments that most of these lifts (all if $k>0,j>3$) will be weakly endoscopic in the following sense:
\begin{Def}
A cuspidal automorphic representation $\pi$ of $G$ is called weakly endoscopic if there is a cuspidal automorphic representation $\pi_1 \boxtimes \pi_2$ of $\gl_2(\mathbb{A}) \times \gl_2(\mathbb{A})$ such that the Satake parameters of $\pi$ and $\pi_1 \boxtimes \pi_2$ agree for almost all primes via the inclusion of dual groups $\gl_2 \times_{\mathbb{G}_m}\gl_2 \to \gsp_4$.
\end{Def}
\section{Integral models of Siegel modular threefolds} \label{integral}
In this section we will discuss Siegel threefolds with various parahoric level structures at a fixed prime $p$, define integral models and describe their singularities, following \cite{Tilouine} and \cite{YuParamodular}. We will then study the natural maps between these models and compute the fibers in characteristic $p$. The supersingular locus of the Siegel threefold with good reduction is well understood by classical work of Katsura-Oort \cite{KatsuraOort} (c.f. \cite{KudlaRapoport}). Combining this description with our understanding of the fibers of the natural maps allows us to describe the supersingular loci of the other Siegel threefolds (c.f.\cite{YuSuperSingular}).
\subsection{Moduli functors}
Let $U^p \subset \gsp_4(\mathbb{A}^{p,\infty})$ be a fixed compact open subgroup away from $p$. Following section 2 of \cite{Tilouine}, we define moduli problems for three different parahorics levels at p, which are defined over $\mathbb{Z}_p$.
\begin{itemize}
    \item The moduli functor $F_H$ of hyperspecial level $H$ which parametrizes prime-to-$p$ isogeny classes of triples $(A, \lambda, \eta)$ where $A$ is an abelian scheme of relative dimension two, $\lambda$ is a prime-to-$p$ polarisation and $\eta$ is a $U^p$ level structure (c.f. \cite[Definition 1.3.7.1.]{LanThesis}).
        
    \item The moduli functor $F_K$ of paramodular level $K$ which parametrizes prime-to-$p$ isogeny classes of triples $(A, \lambda, \eta)$ where $A$ is an abelian scheme of relative dimension two, $\lambda$ is a polarisation such that $\ker \lambda \subset A[p]$ has rank $p^2$, and $\eta$ is a $U^p$ level structure.
        
    \item The moduli functor $F_Q$ of Klingen level $Q$ which parametrizes prime-to-$p$ isogeny classes of quadruples $(A, \lambda, \eta, H)$ where $A$ is an abelian scheme of relative dimension two, $\lambda$ is a prime-to-$p$ polarisation, $\eta$ is a $U^p$ level structure and $H \subset A[p]$ is a finite locally free group scheme of rank $p$.
\end{itemize}
There are natural maps
\begin{equation}
    \begin{tikzcd}
    & F_{Q} \arrow{dr}{b} \arrow{dl}[swap]{a} \\
    F_H && F_K
    \end{tikzcd}
\end{equation}
where $a$ is the forgetful map and $b$ is the map that takes $A$ to $A/H$ with the induced polarisation and level structure. To elaborate, a sub-group scheme $H \subset A[p]$ of rank $p$ is automatically isotropic for the Weil pairing induced by $\lambda$ and hence the polarisation $p \cdot \lambda$ descends uniquely to a polarisation on $A/H$ (c.f. Proposition 11.25 of \cite{MoonenAV}). The following result is standard (c.f. \cite[Theorem 1.4.1.11]{LanThesis})
\begin{Prop}
If $U$ is neat (c.f. \cite[Definition 1.4.1.8]{LanThesis}), then the functors $F_{\star}$ are representable by quasi-projective $\mathbb{Z}_p$ schemes $A_H, A_Q, A_K$. The first one is smooth while the latter two have smooth generic fibres. Moreover the maps $a$ and $b$ are proper, finite étale on the generic fibre and induce bijections on geometrically connected components.
\end{Prop}
\subsection{Singularities}
In this section we will describe the bad reduction of $A_{K}$ and $A_Q$ following \cite{Tilouine} and \cite{YuParamodular}.

\begin{Prop}(Theorem 3 of \cite{Tilouine})
The geometric special fiber of the Siegel modular threefold of Klingen level decomposes as $\aksp=X^{\text{e}} \cup X^{\text{m}}$, where $X^{\text{e}}, X^{\text{m}}$ are smooth and intersect in a smooth surface $S$. Both $X^{\text{e}}$ and $X^{\text{m}}$ have a unique connected component lying over each connected component of $\ahsp$. The locus $X^{\text{e}} \setminus X^{\text{m}}$ parametrizes those quadruples where $H$ is \'etale, the locus $X^{\text{m}} \setminus X^{\text{e}}$ parametrizes those quadruples where $H$ is multiplicative and $S$ parametrizes the quadruples where $H$ is \'etale locally isomorphic to $\alpha_p$.
\end{Prop}
\begin{Prop}(Theorem 1.3 of \cite{YuParamodular})
The scheme $A_{K} / \mathbb{Z}_p$ is regular, and it is smooth over $\mathbb{Z}_p$ away from a finite set of closed points $\Sigma$ in the special fiber. At the singular points in the geometric special fibre $\apsp$ the completed local ring is isomorphic to
\begin{align}
    \overline{\mathbb{F}}_p \llbracket X,Y,Z,W \rrbracket/(XY-ZW).
\end{align}
The singular points $\Sigma$ correspond precisely to the polarized abelian surfaces $(A, \lambda)$ where $\ker \lambda$ is etale locally isomorphic to $\alpha_p \times \alpha_p$. Moreover for every connected component of $\apgen$ there is a unique connected component of $\apsp$.
\end{Prop}

\subsection{Geometry of Hecke correspondences}
The following two lemmata can be proven using the moduli descriptions and Section 4 of \cite{YuSuperSingular}, see also Section 6.3 of \cite{ShenYuZhang}.
\begin{Lem} \label{fibersofalemma}
Let $(A, \lambda, \eta)$ be an $\overline{\mathbb{F}}_p$ point of $\ahsp$, let $a^{-1}(A, \lambda,\eta)$ denote the scheme theoretic fiber. The underlying reduced scheme of $a^{-1}(A, \lambda, \eta)$ has: cardinality $2(p+1)$ if $A$ is ordinary, cardinality $3$ if $A$ has $p$-rank one, cardinality $1$ when $A$ is supersingular but not superspecial. If $A$ is superspecial then $a^{-1}(A, \lambda,\eta)$ is isomorphic to $\mathbb{P}^1$.
\end{Lem}
\begin{Lem}
Let $(A, \lambda, \eta)$ be an $\overline{\mathbb{F}}_p$ point of $\apsp$ and let $b^{-1}(A, \lambda,\eta)$ denote the scheme-theoretic fiber. We will describe the underlying reduced scheme of $b^{-1}(A, \lambda,\eta)$ for various $(A, \lambda,\eta)$. It consists of two points when $A$ has $p$-rank two, one point when $A$ has $p$-rank one and $\ker \lambda$ is étale-multiplicative, two points when $A$ has $p$-rank one and $\ker \lambda$ is local-local, and it consists of a single point when $A$ is supersingular but $(A, \lambda,\eta) \not \in \Sigma$ (recall that $\Sigma$ is the singular locus of $\apsp$). For $(A, \lambda,\eta) \in \Sigma$ the fiber $b^{-1}(A, \lambda, \eta)$ is isomorphic to $\mathbb{P}^1$. 
\label{fibersblemma}
\end{Lem}
\begin{proof}
Let $(A, \lambda, \eta)$ as in the lemma. We want to find all principally polarised abelian varieties $(B, \mu)$ with a degree $p$ isogenies $\phi:B \to A$ such that $\ker \phi \subset B[p]$ has degree $p$ and such that the following diagram commutes
\begin{equation}
    \begin{tikzcd}
    A \arrow{r}{\lambda}  & A^t \arrow{d}{\phi^t} \\
    B \arrow{u}{\phi} \arrow{r}{p \mu}& B^t. 
    \end{tikzcd}
\end{equation}
There is a unique isogeny $\lambda^t:A^t \to A$ such that $\lambda^t \circ \lambda = [p]$ and the choice of $(B, \mu)$ is equivalent to the choice of degree $p$ subgroup of $\ker \lambda^t$. The lemma now follows from a case-by-case analysis; for the sake of brevity we only highlight two cases (c.f. Theorem 4.7 of \cites{YuSuperSingular} for the supersingular cases).  
\begin{itemize}
    \item If $A$ has $p$-rank one and $\ker \lambda$ is local-local then $\ker \lambda^t \simeq \mu_p \times Z/p\mathbb{Z}$ and there are two degree $p$ subgroups.
    
    \item If $A$ has $p$-rank one and $\ker \lambda$ is étale-multiplicative then $\ker \lambda^t \simeq E[p]$ (the $p$-torsion of a supersingular elliptic curve), and so there is only one degree $p$ subgroup.
\end{itemize}
\end{proof}
\subsection{Supersingular loci} \label{supersingularloci}

In this section we will give a description of the supersingular loci of $\aksp, \ahsp$ and $\apsp$, fol owing \cite{KudlaRapoport} and \cite{YuSuperSingular}. It is a classical result due to Katsura-Oort \cite{KatsuraOort} that all the irreducible components of the supersingular locus of the (coarse) moduli space of principally polarized abelian surfaces are isomorphic to $\mathbb{P}^1$. In Chapter 4 of \cite{KudlaRapoport}, it is shown that the supersingular locus of $\ahsp$ is a union of projective lines and that the irreducible components are in bijection with
\begin{align}
N=G(\qqq) \setminus G(\mathbb{A}_f)/ U^p K_2,
\end{align}
equivariant from the prime-to-$p$ Hecke operators. Our computation of the fibers of $a$ now tells us that the irreducible components of the supersingular locus $\aksp$ are parametrised by
\begin{align}
    N \coprod M,
\end{align}
where $M$ is the set of superspecial points in $\ahsp(\overline{\mathbb{F}}_p)$. We will later write
\begin{align}
  \aksp^{\text{ss}}=E \cup F \label{defofE}
\end{align}
where $E$ corresponds to $N$ and $F$ corresponds to $M$. In \cite{YuSuperSingular}, Yu proves the following result, which will be important for us in the proof of our main theorem.
\begin{Prop}(Theorem 1.2 of \cite{YuSuperSingular}) \label{YuThm}
The irreducible components of $E$ are pairwise disjoint, and so are the irreducible components of $F$. The map $b:\aksp \to \apsp$ contracts the irreducible components of $E$ to points, identifying the singular points of $\apsp$ with $N$. 
\end{Prop}
\begin{Rem}
The bijection $\Sigma \simeq N$ can also be proven directly, using the fact that all superspecial abelian surfaces are isomorphic; then $\Sigma$ corresponds to choices of a polarisation $\lambda$ and a level structure $\eta$ on a fixed superspecial abelian surface $A$.
\end{Rem}
\section{The Picard-Lefschetz formula} \label{nearbycycles}
In this section we discuss the Picard-Lefschetz formula for $A_K$, following \cite{SGA7} Expos\'e XV and using results of \cite{LanStroh}. Let $\ell \not=p$, let $\mathcal{O}$ denote $\mathbb{Z}_{\ell}$ and let $\mathbb{V}$ be an automorphic local system of $\mathcal{O}$-modules on $A_K$, and recall that there is a distinguished triangle on the special fiber $\apsp$
\begin{align} \label{fundamentaltriangle}
    \mathbb{V} \to R \Psi \mathbb{V} \to R \Phi \mathbb{V}
\end{align}
where $R \Psi \mathbb{V}$ is the complex of nearby cycles and $R \Psi \mathbb{V}$ is the complex of vanishing cycles. It follows from Corollary 4.6 of \cite{LanStroh} that 
\begin{align} \label{NearbyCycles}
    H^{\bullet}(X_{\overline{s}}, R \Psi \mathbb{V})&=H^{\bullet}(X_{\overline{\eta}}, \mathbb{V}) \\
    H^{\bullet}_c(X_{\overline{s}}, R \Psi \mathbb{V})&=H^{\bullet}_c(X_{\overline{\eta}}, \mathbb{V}).
\end{align}
\subsection{The exact sequence}
Proposition 3.1.2. of Expos\'e XV of \cite{SGA7} then tells us that the sheaves $R^i \Phi \mathbb{V}$ are zero for $i \not=3$ and supported on the singular locus $\Sigma$ for $i=3$. This means in particular that 
\begin{align}
    R^3 \Phi \mathbb{V} = \bigoplus_{x \in \Sigma} i_{x, \ast}\left(R^3 \Phi \mathbb{V} \right)_x
\end{align}
and so the cohomology of $R^3 \Phi \mathbb{V}$ is concentrated in degree zero where we find
\begin{align}
    H^3_c(\apsp, R \Phi \mathbb{V})=H^3(\apsp, R \Phi \mathbb{V})= H^0_c(\apsp, R^3 \Phi \mathbb{V}) = H^0(\apsp, R^3 \Phi \mathbb{V}) = \bigoplus_{x \in \Sigma} \left(R^3 \Phi \mathbb{V}\right)_x.
\end{align}
It follows from \eqref{NearbyCycles} and \eqref{fundamentaltriangle}) that
\begin{align}
    H^i_c(\apsp, \mathbb{V}) & = H^i_c(\apgen, \mathbb{V}) \\
    H^i(\apsp, \mathbb{V}) & = H^i(\apgen, \mathbb{V})
\end{align}
for $i \not=3,4$ and also that we get the following diagram of exact sequences
\begin{equation} 
    \begin{tikzcd}[column sep=small]
        0 \arrow{r} & H^3_c(\apsp, \mathbb{V}) \arrow{r} \arrow{d} &  H^3_c(\apgen, \mathbb{V}) \arrow{r}{\alpha} \arrow{d} & H^3_c(\apsp, R \Phi \mathbb{V}) \arrow{r} \arrow[d, equals]& H^4_c(\apsp, \mathbb{V}) \arrow{r} \arrow{d} &  H^4_c(\apgen, \mathbb{V}) \arrow{r} \arrow{d} & 0 \\
        0 \arrow{r} & H^3(\apsp, \mathbb{V}) \arrow{r} &  H^3(\apgen, \mathbb{V}) \arrow{r}{\alpha} & H^3(\apsp, R \Phi \mathbb{V}) \arrow{r} & H^4(\apsp, \mathbb{V}) \arrow{r} &  H^4(\apgen, \mathbb{V}) \arrow{r} & 0.
    \end{tikzcd} \label{double-exact}
\end{equation}

\subsection{The action of inertia}
Let $\mathbb{Z}_{\ell}(1)$ with generator $\sigma$ be the maximal pro-$\ell$ quotient of the inertia group $I_p$ of $G_{\qqq_p}$, then the general theory tells us that the action of the inertia group factors through $\mathbb{Z}_{\ell}(1)$ and moreover can be computed as follows (c.f. page 33 of Expos\'e XIII of \cite{SGA7}, note that the variation map goes from ordinary cohomology to compactly supported cohomology)
\begin{equation}
    \begin{tikzcd}
        H^3(\apgen, \mathbb{V}) \arrow{r}{\alpha} \arrow{d}{\sigma-1} & \displaystyle \bigoplus_{x \in \Sigma} \left(R^3 \Phi \mathbb{V} \right)_x \arrow{d}{\oplus \operatorname{Var}_x(\sigma)} \\
        H^3_c(\apgen, \mathbb{V})  & \displaystyle \bigoplus_{x \in \Sigma} H^3_x(\apsp, R \Psi \mathbb{V}) \arrow{l}{\beta}
    \end{tikzcd}
\end{equation}
Here $\operatorname{Var}_x(\sigma)$ is the local variation map, which is an isomorphism by 3.3.5 of Exposé XV of \cites{SGA7}, and $H_x(\apsp, -)$ denotes cohomology with support in the closed subscheme $\{x\}$. The map $\sigma-1$ is a priori not equivariant for the action of Frobenius, but we can fix this by defining a monodromy operator $N$ as the logarithm of $\sigma$. Since $(\sigma-1)^2=0$ as an endomorphism of $H^3_{!}(\apgen, \mathbb{V})$, we have $N=\sigma-1$ and so the following diagram is $\frob_p$-equivariant:
\begin{equation}
    \begin{tikzcd}
        H^3_{!}(\apgen, \mathbb{V}) \arrow{r}{\alpha}(1) \arrow{d}{N} & \displaystyle  \bigoplus_{x \in \Sigma} \left(R^3 \Phi \mathbb{V} \right)_x(1) \arrow{d}{\oplus N_x} \\
        H^3_!(\apgen, \mathbb{V})  & \displaystyle  \bigoplus_{x \in \Sigma} H^3_x(\apsp, R \Psi \mathbb{V}) \arrow{l}{\beta}.
    \end{tikzcd}
\end{equation}
Additionally, the general theory tells us that $\beta$ is the Poincare dual of $\alpha$ up to a Tate twist, since $\mathbb{V}$ is self dual up to a Tate twist. Furthermore, the natural map (induced by \eqref{fundamentaltriangle})
\begin{align}
H^3_x(\apsp, \mathbb{V}) \to H^3_x(\apsp, R \Psi \mathbb{V})
\end{align}
is an isomorphism by 2.2.5.8 of Exposé XV of \cites{SGA7}. Taking the direct sum over $x \in \Sigma$ we get an isomorphism
\begin{align}
H^3_{\Sigma}(\apsp, \mathbb{V}) \cong H^3_x(\apsp, R \Psi \mathbb{V})
\end{align}
and under this identification $\beta$ factors as
\begin{equation}
\begin{tikzcd}
H^3_{\Sigma}(\apsp, \mathbb{V}) \arrow{r} \arrow[drr, bend right=40, "\beta", below] \arrow{dr} & H^3_c(\apsp, \mathbb{V}) \arrow{d} \arrow{r} & H^3_c(\apgen, \mathbb{V}) \arrow{d} \\
& H^3(\apsp, \mathbb{V}) \arrow{r} & H^3(\apgen, \mathbb{V})
\end{tikzcd}
\end{equation}
since cohomology with support in a proper subscheme factors through compactly supported cohomology. 
\subsection{The weight filtration}
Now let $\mathbb{V}$ be an automorphic local system of $L=\mathbb{Q}_{\ell}$-vector spaces of pure weight $k$. We will show that the weight filtration on $H^3_!(\apgen, \mathbb{V})$ has nonzero graded pieces only of weights $k+2, k+3, k+4$ by explicitly writing down the weight filtration. Here the weight filtration is defined as usual in terms of the action of $\frob_p$ on $H^3_{!}(\apgen, \mathbb{V})$. 
\begin{Lem} \label{weightmonodromy}
The weight filtration on $H^3_{!}(\apgen, \mathbb{V})$ is given by
\begin{align}
    \im \beta \subset H^3_{!}(\apsp, \mathbb{V}) \subset H^3_{!}(\apgen, \mathbb{V}).
\end{align}
\end{Lem}
\begin{proof}
The quotient $H^3_{!}(\apgen, \mathbb{V})/H^3_{!}(\apsp, \mathbb{V})$ is a subspace of $H^3(\apsp, R \Phi \mathbb{V})$ by \eqref{double-exact}. The derived projection formula tells us that (with tensor products derived)
\begin{align}
    R \Psi \mathbb{V} = \mathbb{V} \otimes_{L} R \Psi L,
\end{align}
where $L/\qqq_{\ell}$ is our coefficient field. This implies that $R \Phi \mathbb{V} = \mathbb{V} \otimes_L R \Phi L$ and the latter is given by
\begin{align}
    R \Phi L \cong \bigoplus_{x \in \Sigma} i_{x, \ast} L(-2)[-3],
\end{align}
by the discussion preceding 2.2.5.7 of Expos\'e XV of \cite{SGA7}(m=1 in our case). This shows that
\begin{align}
    R \Phi \mathbb{V} = \bigoplus_{x \in \Sigma} i_{x, \ast} \mathbb{V}_x(-2)[-3],
\end{align}
and so  $H^3(\apsp, R \Phi \mathbb{V})=\bigoplus_{x \in \Sigma} \mathbb{V}_x(-2)$ is pure of weight $k+4$. We deduce from this that $\im \beta$ is pure of weight $k+2$ using the isomorphism
\begin{align}
    H^3(\apsp, R \Phi \mathbb{V})(1) \cong H^3_{\Sigma}(\apsp, \mathbb{V}),
\end{align}
coming from the local monodromy maps. We are left to show that $H^3_{!}(\apsp, \mathbb{V})/\im \beta$ is pure of weight $k+3$. We know it has weight at most $k+3$ by Theorem 3.3.1 of \cite{WeilII}, so it suffices to show the weight is at least $k+3$. Let $U=\apsp \setminus \Sigma$ be the smooth locus of $\apsp$, then there is a long exact sequence
\begin{equation}
    \begin{tikzcd}[row sep=small]
    &\cdots \arrow{r} & H^3_{\Sigma}(\apsp, \mathbb{V}) \arrow{r}{\beta} & H^3(\apsp, \mathbb{V}) \arrow{r} & H^3(U, \mathbb{V}) \arrow{r} & \cdots
    \end{tikzcd}
\end{equation}
This implies that $H^3_{!}(\apsp, \mathbb{V})/\im \beta$ is contained in $H^3(U, \mathbb{V})$, which has weight at least $k+3$ by smoothness of $U$.
\end{proof}

\subsection{The cokernel of monodromy}
Now let $\mathbb{V}=\mathbb{V}_{\mathcal{O}}$ be an automorphic local system of $\mathcal{O}$-modules, then we can compute the monodromy operator using the following diagram :
\begin{equation} \label{componentgroupdef}
    \begin{tikzcd}
        H^3_!(\apgen, \mathbb{V})(1) \arrow[r,twoheadrightarrow, "\alpha"] \arrow{d}{N} & \im \alpha \arrow[r, hook] \arrow{d}{\gamma} & \displaystyle  \bigoplus_{x \in \Sigma} \left(R^3 \Phi \mathbb{V} \right)(1)_x \arrow{d}{\oplus N_x} \\
        H^3_!(\apgen, \mathbb{V})  & \im \beta \arrow[l, hook] &\displaystyle \bigoplus_{x \in \Sigma} H^3_x(\apsp, R \Psi \mathbb{V}) \arrow[l, twoheadrightarrow, "\beta"].
    \end{tikzcd}
\end{equation}
\begin{Def}
Let $\Theta$ be the cokernel of $\gamma$ (which depends on the weight, the level and $\ell$), we call it the \emph{component group}.
\end{Def}
If we invert $\ell$, then $\operatorname{Im} \alpha$ is the top graded piece of the monodromy filtration and $\operatorname{Im} \beta$ is the bottom graded piece. The weight monodromy conjecture predicts that $\gamma$ is an isomorphism after inverting $\ell$, i.e., that $\Theta$ is torsion. We call $\Theta$ the component group because it is a direct generalisation of the component group of the Jacobian of a modular curve, in its cohomological incarnation (c.f. \cite{Rajaei}). We will show that $\Theta=0$ under the assumptions of Theorem \ref{vanishingA}. The same proof will show that $\gamma$ is an isomorphism after inverting $\ell$, under the assumption that the weight is sufficiently regular (now using Corollary \ref{rationalvanishing}). Showing that $\gamma$ is an isomorphism after inverting $\ell$ in singular weight requires automorphic input.
\section{The component group vanishes} \label{Eisenstein}
\subsection{Statement of the main result}
In this section we will prove that the component group vanishes under the assumptions of Theorem \ref{vanishingA}, it will be useful to give these assumptions a label.
\begin{Assump} \label{assump:one}
Assume that $a > b > 0$, that $U_{\ell}$ is hyperspecial and that $\ell>a+b+4$.
\end{Assump}
\begin{Prop} \label{componentgroupiseisenstein}
If Assumption \ref{assump:one} holds then
\begin{align}
    \Theta=0.
\end{align}
\end{Prop}
The strategy of the proof is to show that the map
\begin{align} \label{mapmap}
    \alpha:H^3_{c}(\apgen, \mathbb{V}) \to H^3_c(\apsp, R \Phi \mathbb{V})
\end{align}
is surjective, which then shows that $\beta$ is injective by duality. Proving this implies the same statement on inner cohomology, because the map from compactly supported cohomology to ordinary cohomology is self dual and because $\alpha$ factors through inner cohomology. The cokernel of \eqref{mapmap} is given by the image of the map
\begin{align}
    H^3_c(\apsp, R \Phi \mathbb{V}) \to H^4_c(\apsp, \mathbb{V}), \label{cokernel}
\end{align}
which has a geometric interpretation in terms of the cycle classes of irreducible components of the supersingular locus of $\ahsp$. To be precise we will identify the image of \eqref{cokernel} with the image of the map
\begin{align}
    H^4_{Z}(\ahsp, \mathbb{V}) \to H^4_c(\ahsp, \mathbb{V}),
\end{align}
where $Z$ is the supersingular locus of $\ahsp$. Now since $\ahsp$ is smooth we know that
\begin{align}
    H^4_c(\ahsp, \mathbb{V}) &= H^4_c(\ahgen, \mathbb{V})
\end{align}
and we can now apply our vanishing theorems to the latter. The proof proceeds in two steps, which we formulate as two lemmas: We first relate the image of \eqref{cokernel} to the cycle class of the supersingular locus on $A_Q$. More precisely we will relate it to the image of the cycle class map $H^0(E, \mathbb{V})(-2) \to H^4(X^e, \mathbb{V})$, where we recall that $\aksp^{\text{ss}}=E \cup F \subset X^{\text{e}}$. Since $E$ and $X^e$ are smooth we can instead rewrite this using cohomology with support in $E$, which will appear in the following lemma.
\begin{Lem}
\label{GeometricVanishingCycles}
There is an isomorphism making the following diagram commute
\begin{equation}
	\begin{tikzcd}
		H^4_E(X^{\text{e}}, \mathbb{V}) \arrow{r} & H^4_c(X^{\text{e}}, \mathbb{V}) \\
		H^3(\apsp, R \Phi \mathbb{V}) \arrow[u, densely dotted] \arrow{r} & H^4_c(\apsp, \mathbb{V}) \arrow{u}{j}
	\end{tikzcd}
\end{equation}
and moreover the natural map $j$ is an isomorphism.
\end{Lem}
Next we relate this to the cycle class map for the supersingular locus $Z$ of $\ahsp$. Because $a_1^{-1}(Z)=E \cup F$ (where $a_1:X^{\text{e}} \to \ahsp$ is the map induced by $a$) there is no natural map $H^4_Z(\ahsp, \mathbb{V}) \to H^4_E(X^{\text{e}}, \mathbb{V})$, however we can produce one as follows:
\begin{Lem} \label{technicallemma}
The natural map $H^4_{Z}(\ahsp, \mathbb{V}) \to H^4_{E \cup F}(X^{\text{e}}, \mathbb{V})$ factors through $H^4_{E}(X^{\text{e}}, \mathbb{V})$ via an isomorphism. To be precise, there is an isomorphism making the following diagram commute:
\begin{equation}
    \begin{tikzcd}
        H^4_{Z}(\ahsp, \mathbb{V}) \arrow{r} \arrow[dr, densely dotted, "\sim"] & H^4_{E \cup F}(X^{\text{e}}, \mathbb{V})  \\
        & H^4_{E}(X^{\text{e}}, \mathbb{V}). \arrow{u}
    \end{tikzcd}
\end{equation}
Here $Z$ is the supersingular locus of $\ahsp$ and $E \cup F$ is the supersingular locus of $\aksp$.
\end{Lem}
\begin{proof}[Proof of Proposition \ref{componentgroupiseisenstein}]
Lemma \ref{GeometricVanishingCycles} and Lemma \ref{technicallemma} give us a commutative diagram
\begin{equation}
    \begin{tikzcd}
         H^3(\apsp, R \Phi \mathbb{V}) \arrow{r} \arrow[d, "\sim"] & H^4_c(\apsp, \mathbb{V}) \arrow[d, "\sim"] \\
        H^4_{E}(X^{\text{e}}, \mathbb{V}) \arrow{r} & H^4_c(X^{\text{e}}, \mathbb{V})  \\
        H^4_{Z}(\ahsp, \mathbb{V}) \arrow[u, "\sim"] \arrow{r} & H^4_c(\ahsp, \mathbb{V}) \arrow[u]
    \end{tikzcd}
\end{equation}
To show that the top map is zero it suffices to show that the bottom map is zero. Now we note that
\begin{align}
    H^4_c(\ahsp, \mathbb{V})=H^4_c(\ahgen, \mathbb{V}),
\end{align}
by smoothness of $A_H$. Moreover by Assumption \ref{assump:one} we can apply Theorem \ref{vanishingA} to deduce that $H^4_c(\ahgen, \mathbb{V})=0$. We have shown that the map $\zeta$ in the following exact sequence is zero
\begin{equation} 
    \begin{tikzcd}
        H^3_c(\apgen, \mathbb{V}) \arrow{r}{\alpha} & H^3(\apsp, R \Phi \mathbb{V}) \arrow{r}{\zeta} & H^4_c(\apsp, \mathbb{V})
     \end{tikzcd}
\end{equation}
which shows that $\alpha$ is surjective and by Poincaré duality we conclude that $\beta$ is injective. This shows that the component group is zero, because the local monodromy operators $N_x$ in \eqref{componentgroupdef} are isomorphisms.
\end{proof}   
\begin{Cor} \label{freeness}
If Assumption \ref{assump:one} holds then the map
\begin{align}
    \beta \otimes \mathbb{F}: H^3_{\Sigma}(\apsp, \mathbb{V}  )  \otimes_{\mathcal{O}} \mathbb{F} \to H^3_!(\apgen, \mathbb{V}  ) \otimes_{\mathcal{O}} \mathbb{F}
\end{align}
is injective.
\end{Cor}
\begin{proof}
Theorem \ref{vanishingA} says that
\begin{align}
    H^3_{!}(\apgen, \mathbb{V}) \otimes_{\mathcal{O}} \mathbb{F} \to H^3_{!}(\apgen, \mathbb{V} \otimes_{\mathcal{O}} \mathbb{F})
\end{align}
is surjective. Moreover since $R \Phi \mathbb{V}$ only has cohomology in degree $3$ we find that
\begin{align}
    H^3(\apsp, R \Phi \mathbb{V}) \otimes_{\mathcal{O}} \mathbb{F} \to H^3(\apsp, R \Phi \mathbb{V} \otimes_{\mathcal{O}} \mathbb{F} )
\end{align}
is an isomorphism. Now these natural maps fit into a commutative cube (which we won't draw) with the maps
\begin{align}
\overline{\alpha}:H^3_!(\apgen, \mathbb{V} \otimes_{\mathcal{O}} \mathbb{F}) &\to H^3(\apsp, R \Phi \mathbb{V} \otimes_{\mathcal{O}} \mathbb{F}) \\
\overline{\beta}:H^3_{\Sigma}(\apsp, \mathbb{V} \otimes_{\mathcal{O}} \mathbb{F}) &\to H^3_!(\apgen, \mathbb{V} \otimes_{\mathcal{O}} \mathbb{F}).
\end{align}
The fact that $\alpha \otimes \mathbb{F}$ is surjective implies that $\overline{\alpha}$ is surjective. By duality we find that $\overline{\beta}$ is injective which shows that $\beta \otimes \mathbb{F}$ is injective by chasing the cube.
\end{proof}
\begin{Cor} \label{rationalprop}
Let $\mathbb{V}=\mathbb{V}_{a,b}$ be a local system of $L$-vector spaces and suppose that the weight is sufficiently regular, i.e., that $a>b>0$. Then the weight-monodromy conjecture holds for 
\begin{align}
\alpha:H^3_{!}(\apgen, \mathbb{V}), 
\end{align}
in other words Theorem \ref{WMtheorem} holds. 
\end{Cor}
\begin{proof}
The proof is the same as the proof of Proposition \ref{componentgroupiseisenstein}, except that we note that the vanishing of $H^4_c(\ahgen, \mathbb{V})$ holds without restrictions on $\ell$  (this is Corollary \ref{rationalvanishing}).
\end{proof}
\subsection{Proof of Lemma \ref{GeometricVanishingCycles}}
\begin{proof}[Proof of Lemma \ref{GeometricVanishingCycles}]
We apply the functors $R \Gamma_c$ and $R \Gamma_{\Sigma}$ to the triangle $\mathbb{V} \to R \Psi \mathbb{V} \to R \Phi \mathbb{V}$ to get a commutative diagram with exact rows:
\begin{equation}
    \begin{tikzcd}
    H^3_{\Sigma}(\apsp, R \Psi \mathbb{V}) \arrow{d} \arrow{r}{h} &  H^3_{\Sigma}(\apsp, R \Phi \mathbb{V}) \arrow{d}{f} \arrow{r}{g} & H^4_{\Sigma}(\apsp, \mathbb{V})  \arrow{d} \arrow{r} & H^4_{\Sigma}(\apsp, R \Psi \mathbb{V}) \arrow{d} \\
    H^3_c(\apsp, R \Psi \mathbb{V}) \arrow{r} & H^3_c(\apsp, R \Phi \mathbb{V}) \arrow{r} \arrow[d, equal] \ & H^4_c(\apsp, \mathbb{V}) \arrow{r} & H^4_c(\apsp, R \Psi \mathbb{V}) \\
    & H^3(\apsp, R \Phi \mathbb{V}).
    \end{tikzcd}
\end{equation}
The map $h$ is zero by 2.2.5.8 of Expos\'e of \cite{SGA7} showing that $g$ is injective. The group $H^4_{\Sigma}(\apsp, R \psi \mathbb{V})$ vanishes by the discussion preceding op. cit. 2.2.5.1, showing that $g$ is an isomorphism. Since the sheaf $R \Phi \mathbb{V}$ is supported on $\Sigma$ we conclude that $f$ is an isomorphism using the long exact sequence
\begin{align}
    \cdots \to H^2(U, R \Phi \mathbb{V}) \to H^3_{\Sigma}(\apsp, R \Phi \mathbb{V}) \to H^3(\apsp, R \Phi \mathbb{V}) \to H^3(U, R \Phi \mathbb{V}) \to \cdots
\end{align}
Recall that $b_1:X^e \to \apsp$ satisfies $b_1^{-1}(\Sigma)=E$ by Proposition \ref{YuThm}. This means that we get a morphism of cohomology with supports $H^4_{\Sigma}(\apsp, \mathbb{V}) \to H^4_E(X^{\text{e}}, \mathbb{V})$ which fits into the following diagram:
\begin{equation}
    \begin{tikzcd}
    H^3(\apsp, R \Phi \mathbb{V}) \arrow[r, "\cong"] & H^3_{\Sigma}(\apsp, R \Phi \mathbb{V}) \arrow[r, "\cong"] & H^4_{\Sigma}(\apsp, \mathbb{V}) \arrow{r} \arrow{d} & H^4_c(\apsp, \mathbb{V}) \arrow{d}{j} \\
    & & H^4_{E}(X^{\text{e}}, \mathbb{V}) \arrow{r} & H^4_c(X^{\text{e}}, \mathbb{V})
    \end{tikzcd}
\end{equation}
The map $\pi=b_1$ induces a map $\mathbb{V} \to R \pi_{\ast} \pi^{\ast} \mathbb{V}$ which we extend to a distinguished triangle
\begin{align}
    \mathbb{V} \to R \pi_{\ast} \pi^{\ast}\mathbb{V} \to Q. \label{distinguished2}
\end{align}
Recall that $b_1:X^e \to \apsp$ has fibers isomorphic to singletons over all points, except that over the singular points of $\apsp$ the fibers are isomorphic to $\mathbb{P}^1$ (Lemma \ref{fibersblemma}). Using the proper base change theorem and the fact that the pullback of $\mathbb{V}$ to $E_x \cong \mathbb{P}^1$ is constant and isomorphic to $\mathbb{V}_x$ we find that
\begin{align} \label{fibercohomology}
    Q \cong \bigoplus_{x \in \Sigma} i_{x, \ast} \mathbb{V}_x(-1)[-2]
\end{align}
where $i_x: \{x \} \to X$ is the inclusion of the point $x$. This means in particular that $H^{\bullet}_{\Sigma}(\apsp, Q)=H^{\bullet}(\apsp, Q)$ since the restriction of $Q$ to $U$ is zero. If we now apply $R \Gamma_{\Sigma}$ to \eqref{distinguished2} we get a long exact sequence (using the fact that $R \Gamma_{\Sigma} R \pi_{\ast} \cong R \Gamma_{\pi^{-1}(\Sigma)}$)
\begin{align}
    \cdots \to H^i_{\Sigma}(\apsp, \mathbb{V}) \to H^i_{E}(X^e, \mathbb{V}) \to H^i_{\Sigma}(\apsp, Q) \to \cdots.
\end{align}
Because the cohomology of $Q$ is concentrated in degree $2$ we see that the map $H^4_{\Sigma}(\apsp, \mathbb{V}) \to H^4_{E}(X^{\text{e}}, \mathbb{V})$ is an isomorphism. By a similar long exact sequence argument (applying $R \Gamma$ to \eqref{distinguished2}) it follows that $j$ is an isomorphism.
\end{proof}

\subsection{Proof of Lemma \ref{technicallemma}}
The lemma will follow from the fact that $a_1:E \to Z$ induces a bijection on irreducible components and the fact that the sheaf $\mathbb{V}$ is constant on $E$ and $Z$ (because the components are projective lines). However, since there is no natural map $H^4_{Z}(\ahsp, \mathbb{V}) \to H^4_E(X^{\text{e}}, \mathbb{V})$ we have to pass to the smooth locus of $Z$ which makes the proof look more involved than it is.
\begin{proof}[Proof of Lemma \ref{technicallemma}]
Let $V=\ahsp \setminus Z$ be the complement of $Z$, let $U$ be the complement of the singular locus of $Z$ and let $V'=X^{\text{e}} \setminus \left(E \cup F \right)$. Then we have a commutative diagram of maps of pairs (recall that $a(F)$ is precisely the set of singular points of $Z$)
\begin{equation}
    \begin{tikzcd}
    &  (X^{\text{e}}, X^{\text{e}} \setminus E)  \\
    (X^{\text{e}} \setminus F, V') \arrow{rr} \arrow{ur}  \arrow{d} &  &  (X^{\text{e}}, V') \arrow{ul} \arrow{d} 
    \\ (U,V) \arrow{rr} & & (\ahsp,V). 
    \end{tikzcd}
\end{equation}
There are induced (contravariant) maps in cohomology with support
\begin{equation}
    \begin{tikzcd}
    & H^4_{E}(X^{\text{e}}, \mathbb{V}) \arrow{dl} \arrow{dr}{\psi} \\
    H^4_{E \cup F}(X^{\text{e}}, \mathbb{V}) \arrow{rr} &  & H^4_{E \setminus F}(X^{\text{e}} \setminus F, \mathbb{V}) \\
    H^4_{Z}(\ahsp, \mathbb{V}) \arrow{u} \arrow{rr}{\phi} & &  H^4_{Z^{\text{sm}}}(U, \mathbb{V})  \arrow{u}{\tau}
    \end{tikzcd}
\end{equation}
and we are going to show that $\phi, \psi$ and $\tau$ are all isomorphisms, providing the required factorisation. The map $\phi$ fits into a long exact sequence for $\ahsp \supset U \supset V$ (c.f. Chapter 23 of \cite{milneLEC})
\begin{align}
    \cdots \to H^4_{Z^{\text{sing}}}(\ahsp, \mathbb{V}) \to H^4_{Z}(\ahsp, \mathbb{V}) \to H^4_{Z^{\text{sm}}}(U, \mathbb{V}) \to H^5_{Z^{\text{sing}}}(\ahsp, \mathbb{V}) \to \cdots
\end{align}
Since $Z^{\text{sing}}$ is smooth of codimension $3$ in $\ahsp$, cohomological purity (Theorem 16.1 of \cite{milneLEC}) tells us that
\begin{align}
    H^i_{Z^{\text{sing}}}(\ahsp, \mathbb{V})=H^{i-6}(Z^{\text{sing}}, \mathbb{V})(-3)
\end{align}
which is zero for $i=4,5$, so $\phi$ is an isomorphism. Now we note that both $Z^{\text{sm}}$ and $E \setminus F$ are smooth of codimension $2$ in $U$ and $X^{\text{e}} \setminus F$ respectively so we can use cohomological purity twice to produce a commutative diagram
\begin{equation}
    \begin{tikzcd}
        H^4_{E \setminus F}(X^{\text{e}} \setminus F, \mathbb{V}) \arrow[r, "\sim"] & H^0(E \setminus F, \mathbb{V})(-2) \\
        H^4_{Z^{\text{sm}}}(U, \mathbb{V})  \arrow{u}{\tau} \arrow[r, "\sim"]& H^0(Z^{\text{sm}}, \mathbb{V})(-2) \arrow{u}{\sigma}.
    \end{tikzcd}
\end{equation}
The map $\sigma$ is an isomorphism because the map $E \setminus F \to Z^{\text{sm}}$ induces a bijection on irreducible components (Section \ref{supersingularloci}) and the sheaf $\mathbb{V}$ is constant on each of those components (because the components are isomorphic to $\mathbb{P}^1$). Similarly, we get a commutative diagram
\begin{equation}
    \begin{tikzcd}
    H^4_E(X^e, \mathbb{V}) \arrow{d}{\psi} \arrow[r, "\sim"] & H^0(E, \mathbb{V})(-2) \arrow{d}{\xi} \\
    H^4_{E \setminus F}(X^e \setminus F, \mathbb{V}) \arrow[r, "\sim"] & H^0(E \setminus F, \mathbb{V})(-2),
    \end{tikzcd}
\end{equation}
and $\xi$ is an isomorphism because $(E \setminus F) \xhookrightarrow{} E$ induces a bijection on irreducible components and $\mathbb{V}$ is constant on $E$ because $E \simeq \mathbb{P}^1$.
\end{proof}
\section{Proofs of the main results} \label{applications}

\subsection{Mazur's principle}
Let us recall the theorem that we are trying to prove:
\mazursprinciple*
\begin{proof}
Assume that there are no congruences to unramified cuspidal automorphic representations $\pi'$, in other words, that the conclusion of the theorem does not hold. We will prove that $\overline{\rho_{\pi, \ell}}$ has at most three Frobenius eigenvalues, showing that the assumptions of the theorem are not satisfied. Let $\mathbb{T}$ be the Hecke algebra containing the unramified Hecke operators and the Hecke operators at $p$ and let $\mathfrak{m}$ be the maximal ideal of the Hecke algebra corresponding to $\overline{\rho_{\pi, \ell}}$, then we have a decomposition
\begin{align}
    H^3_{!}(\apgen, \mathbb{V})_{\mathfrak{m}} \otimes_{\mathcal{O}} L= \bigoplus_{} \rho_{\pi',\ell} \label{decompositionfive},
\end{align}
where $\rho_{\pi',\ell}$ is congruent modulo $\ell$ to $\rho_{\pi, \ell}$ and might appear multiple times in the direct sum. Note that Yoshida lifts and Saito-Kurokawa lifts don't contribute to the direct sum because the associated Galois representations are reducible and so cannot be congruent to $\rho_{\pi, \ell}$. This means that $\eqref{decompositionfive}$ is really a direct sum of irreducible four-dimensional Galois representations. By assumption all the $\pi'_p$ over which the sum is indexed are ramified at $p$ and so by weight-monodromy (Corollary \ref{rationalprop}) and local-global compatibility we know that $\pi'_p$ is of type IIa. This means that the monodromy operator $N:\rho_{\pi',\ell}(1) \to \rho_{\pi',\ell}$ has a one-dimensional image, which shows that
\begin{align}
    \alpha(\rho_{\pi',\ell})
\end{align}
is a one-dimensional subspace of $H^3(\apsp, R \Phi \mathbb{V})_{\mathfrak{m}} \otimes_{\mathcal{O}} L$. Since $\alpha$ is moreover surjective we find that the dimension of $H^3(\apsp, R \Phi \mathbb{V})_{\mathfrak{m}} \otimes_{\mathcal{O}} L$ is equal to $n$, the number of $\rho_{\pi', \ell}$'s appearing in the direct sum \eqref{decompositionfive}.

By Lemma \ref{eigenvaluelemma} we know that Frobenius acts on $\alpha(\rho_{\pi',\ell})$ by the scalar $p \lambda_{\pi'}$, where $\lambda_{\pi'}$ is the eigenvalue of the $u$-operator on $(\pi'_p)^{K(p)}$, introduced in Section \ref{Atkin-Lehner}. Localising at $\mathfrak{m}$ means fixing a mod $\ell$ eigenvalue of $u$ and so Frobenius acts with a single eigenvalue on
\begin{align}
    H^3(\apsp, R \Phi \mathbb{V})_{\mathfrak{m}} \otimes_{\mathcal{O}} \mathbb{F}.
\end{align}
We also know that $H^3_{!}(\apgen, \mathbb{V})_{\mathfrak{m}} \otimes_{\mathcal{O}} \mathbb{F}$ is isomorphic (up to semi-simplification) to
\begin{align}
    \overline{\rho_{\pi, \ell}}^{\oplus n}
\end{align}
for the same $n$ as before (by vanishing of torsion). Since $\overline{\rho_{\pi, \ell}}$ has four distinct Frobenius eigenvalues, we know that the space $H^3_{!}(\apgen, \mathbb{V})_{\mathfrak{m}} \otimes_{\mathcal{O}} \mathbb{F}$ decomposes into four generalised Frobenius eigenspaces, each of dimension $n$. Now $\alpha \otimes \mathbb{F}$ is surjective because $\alpha$ is surjective hence the kernel of $\alpha \otimes \mathbb{F}$ only contains three Frobenius eigenvalues. Because $\overline{\rho_{\pi, \ell}}$ is irreducible we can find a copy
\begin{align}
    \overline{\rho_{\pi, \ell}} \subset H^3_{!}(\apgen, \mathbb{V})_{\mathfrak{m}} \otimes_{\mathcal{O}} \mathbb{F},
\end{align}
which is contained in the kernel of $N \otimes \mathbb{F}$ since $\overline{\rho_{\pi, \ell}}$ is unramified. As usual we compute the monodromy operator using the following diagram:
\begin{equation}
    \begin{tikzcd}
       H^3_{!}(\apgen, \mathbb{V})(1)_{\mathfrak{m}} \otimes_{\mathcal{O}} \mathbb{F} \arrow[r, "\alpha \otimes \mathbb{F}", twoheadrightarrow] \arrow{d}{N \otimes \mathbb{F}} & H^3(\apsp, R \Phi \mathbb{V})(1)_{\mathfrak{m}} \otimes_{\mathcal{O}} \mathbb{F} \arrow{d} \\        
       H^3_!(\apgen, \mathbb{V})_{\mathfrak{m}} \otimes_{\mathcal{O}} \mathbb{F}  &  H^3_{\Sigma}(\apsp, \mathbb{V})_{\mathfrak{m}} \otimes_{\mathcal{O}} \mathbb{F} \arrow[l, "\beta \otimes \mathbb{F}", hook]
    \end{tikzcd}
\end{equation}
Corollary \ref{freeness} tells us that $\beta \otimes \mathbb{F}$ is injective which means that $\ker (N \otimes \mathbb{F}) = \ker (\alpha \otimes \mathbb{F})$. Therefore our copy of $\overline{\rho_{\pi, \ell}}$ must be contained in the kernel of $\alpha \otimes \mathbb{F}$, but then $\overline{\rho_{\pi, \ell}}$ only has three distinct Frobenius eigenvalues, which gives a contradiction. \end{proof}
\subsection{A geometric Jacquet-Langlands correspondence}
Let us start by stating a precise version of Theorem \ref{JLtheorem}:
\begin{Thm} \label{JLtheorem2}
\begin{enumerate}[label=(\arabic*)]
\item Let $\pi$ be a cohomological cuspidal automorphic representation of $\gsp_4$ that is not an irrelevant Yoshida lift and such that $\pi_p$ is ramified and $K(p)$-spherical. Then there is a cuspidal automorphic representation $\sigma$ of $G$ such that $\pi_v \cong \sigma_v$ for finite places $v \not=p$, such that $\sigma_p$ is $K_2(p)$-spherical and with $\sigma_{\infty}$ determined by $\pi_{\infty}$. Moreover, $\sigma$ occurs with multiplicity one in the cuspidal spectrum of $G$. Conversely, a given cuspidal automorphic representation $\sigma$ of $G$ comes from such a $\pi$ if $\sigma_{\infty}$ has weight $k>0,j>3$.
\item Let $k \ge 0, j \ge 3$ and let $N$ be a squarefree integer such that $p \mid N$, then there is an injective map
\begin{align}
S_{k,j}[\para(N)]^{p-\text{new}} \xhookrightarrow{} \mathcal{A}^G_{k,j}[K_2(N)], \label{eq:beta}
\end{align}
equivariant for the prime-to-$p$ Hecke operators, which proves Conjecture \ref{Conj1}. For $k>0, j>3$ the image consists precisely of algebraic modular forms that are not weakly-endoscopic. 
\item Let $k \ge 0, j \ge 3$ and $N=p$ be prime, then there is an injective lift
\begin{align}
    S_{2j-2+k}[\Gamma_0(1)] \times S_{k+2}[\Gamma_0(p)]^{\text{new}} \to \mathcal{A}^G_{k,j}[K_2(p)]
\end{align}
and for $k=0$ an injective lift
\begin{align}
    S_{2j-2}[\Gamma_0(1)] \to \mathcal{A}^G_{0,j}[K_2(p))].
\end{align}
The image of \eqref{eq:beta} is a complementary subspace to the space generated by the images of these lifts (and the constant algebraic modular forms if $k=0,j=3$). Moreover for $k \ge 0, j \ge 3$ we have an equality
\begin{align}
    \dim \mathcal{A}^G_{k,j}[K_2(p)]&=\dim S_{2j-2+k}[\Gamma_0(1)] \cdot \dim S_{k+2}[\Gamma_0(p)]^{\text{new}} + \dim S_{k,j}[\Gamma_0(p)] +  \\ & - 2 \dim S_{k,j}[\Gamma_0(1)] + \delta_{k,0}\dim S_{2j-2}[\Gamma_0(1)] +\delta_{k,0} \cdot \delta_{j,0}.
\end{align}
\end{enumerate}
\end{Thm}
\begin{Rem}
The dimension formula in (iii) is proven by Ibukiyama for $j \not=3,4$ in \cite{IbukiyamaConjecture} and for $k=0,j=3$ in \cite{Ibukiyamadimension}.
\end{Rem}
The main idea of the proof is the incarnation of the singular locus $\Sigma$ of $\apsp$ as a Shimura set for $G$. We then get a map $\alpha$ from $H^3_{!}(\apgen, \mathbb{V})$, which we can explicitly describe in terms of automorphic forms and Galois representations, to the space of algebraic modular forms. Theorem \ref{WMtheorem} combined with local-global compatibility allows us to control the image of $\alpha$ in a Galois-theoretic way. For parts (2) and (3) the main ingredients are the multiplicity one result of \cites{Arthur, GeeTaibi} and the main Theorem of \cite{Petersen}, which computes the cohomology of $\ahgen$ if $U^p=\operatorname{GSp}(\hat{\mathbb{Z}}^p)$. We start by relating spaces of algebraic modular forms to cohomology groups on $\apsp$, using our description of the singular locus:
\begin{Lem}
There is an isomorphism
\begin{align}
    H^0(\Sigma, \mathbb{V}_{a,b}) \cong \mathcal{A}^{G}_{a-b,b+3}[U^p K_2(p)],
\end{align}
equivariant for the prime-to-$p$ Hecke operators.
\end{Lem}
\begin{proof}
This is standard, see the proof of Proposition 6.4 of \cite{Thorne}.
\end{proof}
\begin{proof}[Proof of Theorem \ref{JLtheorem2} (1)]
We can identify the following cohomology groups
\begin{align}
    H^3(\apsp, R \Phi \mathbb{V}) = \bigoplus_{x \in \Sigma} (R^3 \Phi \mathbb{V})_x = \bigoplus_{x \in \Sigma} \mathbb{V}_x(-2)=H^0(\Sigma, \mathbb{V})(-2),
\end{align}
and the key player of the proof will be the Hecke equivariant map
\begin{align}
    \alpha:H^3_{!}(\apgen, \mathbb{V}) \to H^3(\apsp, R \Phi \mathbb{V})=A^{G}_{a-b,b+3}[U^p K_2(p)].
\end{align}
We can write the domain of $\alpha$ in terms of Hecke modules and Galois representations as follows (using the results of Section \ref{Matsushima} and the multiplicity one result of \cites{Arthur,GeeTaibi})
\begin{align}
    H^3_{!}(\apgen, \mathbb{V}) = \bigoplus_{\pi}  \pi_{\text{fin}}^{U} \otimes \rho_{\pi,\ell}.
\end{align}
We are going to hide some things in the notation for bookkeeping purposes: If $\pi$ is a Yoshida lift or CAP then by $ \rho_{\pi,\ell}$ we mean the two-dimensional Galois representation that occurs in the cohomology. If $\pi$ is of general type with $\pi_{\infty}$ holomorphic then $ \rho_{\pi,\ell}$ is the four-dimensional Galois representation associated with $\pi$ and if $\pi_{\infty}$ is generic then $ \rho_{\pi,\ell}=0$. The reason for doing this is that there is only one four-dimensional Galois representation $\rho_{\pi, \ell}$ for the two automorphic representations $\pi_{\text{fin}} \otimes \pi^W$ and $\pi_{\text{fin}} \otimes \pi^H$ as they both contribute a two-dimensional piece to cohomology. 

Because $\alpha$ has an interpretation in terms of the action of inertia, we know that $\rho_{\pi,\ell}/(\ker \alpha \cap \rho_{\pi,\ell})$ is one dimensional for all ramified $ \rho_{\pi,\ell}$ occurring in $H^3_{!}(\apgen, \mathbb{V})$. To be precise this follows from the conjunction of weight-monodromy and local-global compatibility (for weight-monodromy we need to use the results of Arthur to transfer generic cuspidal automorphic representations to $\operatorname{GL}_4$, where weight-monodromy holds by \cites{Caraiani}, c.f. Theorem 2.1.1.(2) of \cites{MR3152941}). Therefore we will write 
\begin{align}
    \im \alpha \cong \bigoplus_{\pi}  \pi_{\text{f}}^{U} \label{eq:decomposition}
\end{align}
where $\pi$ runs over certain cuspidal automorphic representations of $\gsp_4$. Since $\pi_p^{K(p)}$ is one-dimensional we can rewrite this as
\begin{align}
    \im \alpha \cong \bigoplus_{\pi}  (\pi_{\text{f}}^{p})^{U^p}. \label{eq:decomposition2}
\end{align}
Similarly we can write
\begin{align}
    A^{G}_{a-b,b+3}[U^p K_2(p)]=\bigoplus_{\sigma} m(\sigma) \sigma_{\text{f}}^{U'}
\end{align}
where $\sigma$ runs over certain cuspidal automorphic representations of $G$ and $U'= U^p K_2(p)$. Since $\sigma_p^{K_2(p)}$ is one-dimensional we can also write this as
\begin{align}
    A^{G}_{a-b,b+3}[U^p K_2(p)]=\bigoplus_{\sigma} m(\sigma) (\sigma_{\text{f}}^{p})^{U^p}.
\end{align}
Let $\mathbb{T}$ be the Hecke-algebra away from $p$ for both $G$ and $\gsp_4$, which makes sense as soon as we fix an isomorphism $\gsp_4(\mathbb{A}_f^{p}) \cong G(\mathbb{A}_f^{p})$. Then the spaces $(\pi_{\text{f}}^{p})^{U^p}$ and $(\sigma_{f}^{p})^{U^p}$ are simple $\mathbb{T}$-modules and moreover the map $\alpha$ is equivariant for the action of $\mathbb{T}$. 

Now fix a cohomological automorphic representation $\pi$ of $\gsp_4$ that is not an irrelevant Yoshida lift such that $\pi_{\infty}$ is in the discrete series and such that $\pi_p$ is ramified and $K(p)$-spherical. Then we can choose $U^p$ sufficiently small such that $ \rho_{\pi,\ell}$ occurs in $H^3_{!}(\apgen, \mathbb{V}$). Because $\pi$ is not an irrelevant Yoshida lift we know that the summand
\begin{align}
    F_{\pi}:= \pi_{\text{f}}^U \subset \im \alpha
\end{align}
is nonzero (because the Galois representation occurring in cohomology is ramified). There we see that
\begin{align}
    F_{\pi} \subset A^{G}_{a-b,b+3}[U^p K_2(p)]=\bigoplus_{\sigma} m(\sigma) \sigma_{\text{f}}^{U'}
\end{align}
and so there is a cuspidal automorphic representation $\sigma$ of $G$
such that $(\sigma^p)^{U^p} \cong (\pi^p)^{U^p}$ which implies that $\sigma_v \cong \pi_v$ for all finite places $v \not=p$. Moreover $\sigma_{\infty}$ is determined by $\pi_{\infty}$ because it is determined by the weights $a-b,b+3$. When $a>b>0$ or equivalently $k>0,j>3$ then $\alpha$ is surjective by Corollary \ref{rationalprop} so every cuspidal automorphic representation of $G$ arises in this way. To prove multiplicity one for $\sigma$ `in the image' of this transfer we first prove a claim:
\begin{Claim}
The summand $P_{\sigma}:=m(\sigma) \sigma_f^{U'} \subset \mathcal{A}^G_{a-b,b+3}$ is in the image of $\alpha$ (we have only shown so far that $\im \alpha \cap P_{\sigma} \not = \emptyset)$.
\end{Claim}
\begin{proof}[Proof of Claim]
Let $\pi$ be a cuspidal automorphic representation of $\gsp_4$ that `maps to $\sigma$'. If $P_{\sigma}$ maps nontrivially to $\coker \alpha$ then the Hecke eigenvalues of $\sigma$ occur in $H^4_c(\ahgen, \mathbb{V})$ and so the Hecke eigenvalues associated with $\pi$ occur there. Therefore $\pi$ has the same prime-to-$p$ Hecke eigenvalues as an automorphic form $\pi'$ whose Hecke eigenvalues occur in $H^4_c(\ahsp, \mathbb{V})$. By the Chebotarev density theorem this means that the semi-simplification of $\rho_{\pi, \ell}$ is equal to the semi-simplification of $\rho_{\pi', \ell}$. But $\rho_{\pi, \ell}$ is already semi-simple and since $\rho_{\pi', \ell}$ is unramified at $p$ (because $\pi'_{p}$ is unramified), we conclude that $\rho_{\pi, \ell}$ is unramified at $p$, a contradiction.
\end{proof} 
From the claim we get cuspidal automorphic representations $\pi_{1}, \cdots, \pi_{m(\sigma)}$ such that $\alpha(\rho_{\pi_i, \ell})$ maps to the summand $P_{\sigma}$. Therefore we have $\pi_{i,v} \cong \pi_{j,v}$ for all finite places $v \not=p$ and so all the $\pi_i$ are in the same $L$-packet. When the $L$-packet is not CAP, then the $\pi_{i,p}$ are generic (by weight-monodromy) and this means that $\pi_{i,p} \cong \pi_{j,p}$ because local $L$-packets have unique generic constituents. When the $L$-packet is CAP, then the $\pi_i$ are Saito-Kurokawa lifts and the fact that $\pi_{i,p}$ is $K(p)$-spherical also means that $\pi_{i,p} \cong \pi_{j,p}$ (c.f. Table 2 of \cite{SchmidtII}). We conclude that all the $\pi_i$ are isomorphic and therefore by multiplicity one for $\gsp_4$ we deduce that $m(\sigma)=1$.
\end{proof}
\begin{Rem}
In the proof of parts (2),(3) of Theorem \ref{JLtheorem2} we will work with the open compact subgroup $K(N)$, which is not neat. This means that the moduli functors we defined are not representable and so we cannot, strictly speaking, take the \'etale cohomology of the `Shimura variety of level $K(N)$'. To fix this we choose a neat compact open subgroup $U \subset K(N)$ that is normal in $K(N)$ and then define (with similar definitions for compactly supported and inner cohomology)
\begin{align}
    H_{}^{\bullet}(Y_{K(N), \overline{\qqq}}, \mathbb{V}):=H^{\bullet}_{}(Y_{U, \overline{\qqq}}, \mathbb{V})^{H},
\end{align}
where $H$ is the finite group $H=K(N)/U$. Note that because $H$ is finite and $\mathbb{V}$ is a local system of $L$-vector spaces, taking $H$ invariants is exact and so all the exact sequences from the previous sections carry over to this setting. Furthermore, this definition does not depend on the choice of $U$. Moreover the corresponding space of algebraic modular forms $\mathcal{A}^G[U']$ satisfies
\begin{align}
    \mathcal{A}_{k,j}^G[U']^{H}=\mathcal{A}_{k,j}^G[K_2(N)]
\end{align}
for all $k,j$, which means that there is an induced map
\begin{align}
    \alpha: H^{3}_!(\apgen, \mathbb{V}) \to \mathcal{A}_{k,j}^G[K_2(N)].
\end{align}
\end{Rem}
\begin{proof}[Proof of Theorem \ref{JLtheorem2}(2)]
Write
\begin{align}
    H^3_{!}(\akgen, \mathbb{V}):=\bigoplus_{\pi} \pi_{\text{fin}}^U \otimes \rho_{\pi, \ell},
\end{align}
with the notation as in the proof of part (1) of Theorem \ref{JLtheorem2} (so $\rho_{\pi, \ell}$ is either two or four-dimensional). Let $S$ be the subspace of $H^3_{!}(\apgen, \mathbb{V})$ spanned by the summands $\pi_{\text{fin}}^{U} \otimes \rho_{\pi,\ell}$ for $\pi$ such that $\pi_{\infty}$ is holomorphic and such that $\pi_p$ is ramified. Then $\alpha(S)$ has dimension equal to the dimension of $S_{k,j}[\para(N)]^{p-\text{new}}$ which gives us an injective map (after choosing a basis of normalised eigenforms)
\begin{align}
    S_{k,j}[\para(N)]^{p-\text{new}} \to \mathcal{A}^G_{k, j}[K_2(N)] \otimes \mathbb{C}. \label{THEMAP}
\end{align}
Since there are no holomorphic Yoshida lifts, it is clear that the algebraic modular forms in the image are not weakly endoscopic. If $a>b>0$ then $\alpha$ is surjective by Corollary \ref{rationalprop} and the image of \eqref{THEMAP} is complementary to the subspace of weakly endoscopic algebraic modular forms. Indeed, the only cohomological cuspidal automorphic representations of this level that don't come from $S_{k,j}[\para(N)]$ are non-holomorphic Yoshida lifts. \end{proof}
\begin{proof}[Proof of Theorem \ref{JLtheorem2}(3)]
Now suppose that $N=p$, then Theorem 2.1 of \cite{Petersen} tells us that $H^4_c(A_{H, \overline{\qqq}}, \mathbb{V}_{a,b})=0 \label{zero}$ unless $a=b$ is even. This means that $\alpha$ will be surjective unless $a=b$ is even and so we can prove dimension formulas if we understand the dimension of the image of $\alpha$.
\begin{Claim} \label{lastclaim}
The dimension of the image of $\alpha$ is equal to:
\begin{align} \label{formula}
    \dim \im \alpha &= \dim S_{2j-2+k}[\Gamma_0(1)] \times \dim S_{k+2}[\Gamma_0(p)]^{\text{new}} \\ &+ \dim S_{k,j}[K(p)] - 2 \dim S_{k,j}[K(1)] + \delta_{k,0} \tfrac{1+(-1)^j}{2} \dim S_{2j-2}[\Gamma_0(1)].
\end{align}
\end{Claim}
\begin{proof}[Proof of Claim \ref{lastclaim}]
The cohomological cuspidal automorphic representations $\pi$ with nonzero invariants under $K(p)$ are either non-holomorphic Yoshida lifts or correspond to holomorphic Siegel cusp forms. We only care about holomorphic Siegel cusp forms that are new at $p$ and similarly about Yoshida lifts with $\pi_p$ ramified that are not irrelevant. Since $\pi^{K(p)}$ is one-dimensional for all these $\pi$ it suffices to simply count them. The results of \cite{ParamodularForms} tell us that every eigenform $f \in S_{k,j}[K(1)]$ produces two oldforms in $S_{k,j}[K(p)]$, which are distinct unless $k=0$, $j$ is even and $f$ is in the image of the (injective) Saito-Kurokawa lift $S_{2j-2}[\Gamma_0(1)] \to S_{0,j}[K(1)]$. This means that
\begin{align}
    \dim S_{k,j}[K(p)]^{\text{new}}&=\dim S_{k,j}[K(p)]-2 \dim S_{k,j}[K(1)]+\delta_{k,0}\tfrac{1+(-1)^j}{2} \dim S_{2j-2}[\Gamma_0(1)].
\end{align}
Section \ref{Yoshida} tells us that the number of relevant Yoshida lifts is equal to $\dim S_{2j-2+k}[\Gamma_0(1)] \times \dim S_{k+2}[\Gamma_0(p)]^{\text{new}}$ and the formula follows.
\end{proof}
If $a>b$ or if $b$ is odd (equivalently $k>0$ or $j \ge 3$ even) then $\alpha$ is surjective which proves that
\begin{align}
    \dim \mathcal{A}^G_{k, j}[K_2(p)]=\dim \im \alpha = (\star),
\end{align}
where $(\star)$ is given by \eqref{formula}.
When $a=b$ is even and $b>0$ (equivalently $k=0$ and $j \ge 3$ odd) then the dimension formula to be proven follows from Theorem 5.2 of \cite{IbukiyamaConjecture}. The case $a=b=0$ (equivalently $k=0, j=3$) is proven by Ibukiyama in \cite{Ibukiyamadimension} and can also be deduced from the fact that $H^4_c(\ahgen, \mathbb{Q}_{\ell}) = \mathbb{Q}_{\ell}(-2)$.

Now let us return to the case that $a=b$ is even. Theorem 2.1 of \cite{Petersen} tells us that
\begin{align}
    H^4_c(A_{H, \overline{\qqq}}, \mathbb{V}_{a,b}) = L(b-2)^{\oplus s_{a+b+4}}
\end{align}
where $s_{a+b+4}=\dim S_{a+b+4}[\Gamma_0(1)]$. Moreover, the Hecke eigenvalues occurring in $H^4_{c}(A_{H, \overline{\qqq}}, \mathbb{V}_{a,b})=H^4_{c}(\ahgen, \mathbb{V}_{a,b})$ are weakly equivalent to those of automorphic representations parabolically induced from the Siegel parabolic (they are conjecturally non-cuspidal). Theorem 5.2 of \cite{IbukiyamaConjecture} combined with Claim \ref{lastclaim} tells us the dimension of the cokernel of $\alpha$ is also equal to $s_{a+b+4}$, and so
\begin{align}
    H^4_{Z}(\ahgen, \mathbb{V}_{a,b}) \to H^4_{c}(\ahgen, \mathbb{V}_{a,b})
\end{align}
is surjective (recall that $\coker \alpha = \im \gamma$). In any case this means that there is an injective lift $S_{a+b+4}[\Gamma_0(1)] \to \mathcal{A}^G_{0,b+3}[K_2(p)]$ as claimed in the theorem. We conclude that the image of
\begin{align}
    S_{k,j}[\para(p)]^{\text{new}} \to \mathcal{A}^G_{k, j}[K_2(p)]
\end{align}
is a complementary subspace to the subspace generated by the lift from
\begin{align}
    S_{2j-2+k}[\Gamma_0(1)] \times S_{k+2}[\Gamma_0(p)]^{\text{new}},
\end{align}
the lift from $S_{2j-2}[\Gamma_0(1)]$ and the constant algebraic modular forms (if $k=0,j=3$).
\end{proof}
\subsection{Acknowledgements}
I am very grateful to James Newton for his guidance and for carefully reading preliminary versions of this work. I would like to thank Ana Caraiani for her encouragement and her helpful suggestions. I would also like to thank Andrew Graham and Dougal Davis for many useful discussions. Furthermore I would like to thank the anonymous referee for their helpful comments and suggestions.
\appendix 
\section{On the weight-monodromy conjecture}
\label{Appendix:A}
In this appendix we prove the following result (which is presumably well known to experts):
\begin{Lemm} \label{Lemm}
Let $F/\qqq_p$ be a finite extension, let $X/F$ be a Shimura variety of Hodge type, let $j:X \xhookrightarrow{} X^{\ast}$ be the inclusion of $X$ into its minimal compactification and let $\mathbb{V}$ be an automorphic local system on $X$. Then the weight-monodromy conjecture for
\begin{align}
    H^{i}(X^{\ast}_{\overline{F}}, j_{! \ast} \mathbb{V}),
\end{align}
where $j_{! \ast}$ denotes the intermediate extension of perverse sheaves (up to shift), follows from the weight-monodromy conjecture for smooth and proper varieties over $F$
\end{Lemm}
\begin{proof}[Proof of Theorem \ref{Lemm}]
Let $\pi:A \to X$ be the universal family of abelian varieties over $X$ (here we use that $X$ is of Hodge type) and let $\pi^n:A^n \to X$ be the $n$-fold self fiber product of $\pi:A \to X$. Then Proposition 3.2 of \cite{LanStroh} tells us that every automorphic local system $\mathbb{V}$ is a direct summand of $R \pi^n_{\ast} \mathbb{Q}_{\ell}$, up to shift and Tate twist. So it suffices to prove weight monodromy for
\begin{align}
    H^{\bullet}(X_{\overline{F}}^{\ast}, j_{! \ast} R \pi^n_{\ast} \mathbb{Q}_{\ell}).
\end{align}
Let $\sigma: A^{\ast} \to X^{\ast}$ be a projective morphism extending $\pi$ (which exists by \cite{Nagata}) and let $\sigma^n$ be its $n$-fold self fiber product, which sits in the following Cartesian diagram over $\overline{F}$.
\begin{equation}
    \begin{tikzcd}
    A^n_{\overline{F}} \arrow{r}{k} \arrow{d}{\pi^n}& A^{\ast, n}_{\overline{F}} \arrow{d}{\sigma^n} \\
    X_{\overline{F}} \arrow{r}{j} & X_{\overline{F}}^{\ast}
    \end{tikzcd}
\end{equation}
Let us note that $k_{! \ast}\mathbb{Q}_{\ell}[d]$, where $d$ is the dimension of $A^n$, is the intersection cohomology complex $\operatorname{IC}_{A^{\star,n}}$ of $A^{\star, n}$. Theorem 1.8 of \cites{MR3477743} (a refinement of the decomposition theorem) tells us that there is a Galois equivariant decomposition
\begin{align}
    R \sigma^n_\ast \operatorname{IC}_{A^{\star,n}} \simeq \bigoplus_i {}^{\mathfrak{p}}\mathcal{H}^i(R \sigma^n_{\ast} \operatorname{IC}_{A^{\star,n}})[-i].
\end{align}
and similarly
\begin{align}
    R \pi^n_{\ast} \mathbb{Q}_{\ell}[d] \simeq \bigoplus_i  R^i \pi^n \mathbb{Q}_{\ell}[d-i].
\end{align}
\begin{Claim}
For each $i$ there is a Galois-equivariant splitting
\begin{align}
    j_{! \ast} R^i \pi^n_{\ast} \mathbb{Q}_{\ell}[d-i] \subset^{\oplus} {}^{\mathfrak{p}}\mathcal{H}^i(R \sigma^n_{\ast} \operatorname{IC}_{A^{\star,n}})[-i].
\end{align}
\end{Claim}
\begin{proof}
The proper base change theorem tells us that the left hand side and the right hand side have the same restriction to $X_{\overline{F}} \subset X^{\ast}_{\overline{F}}$. Then Lemma 2.2.8 of \cites{MR3477743} gives us the result (rather its Galois equivariance). 
\end{proof}
Lemma 1.4 of \cites{TaylorYoshida} tells us that validity of weight-monodromy passes to direct summands. We compute
\begin{align} \label{eq:DirectSummandEq}
    H^{\bullet}(X_{\overline{F}}^{\ast}, j_{! \ast} R \pi^n_{\ast} \mathbb{Q}_{\ell}) &\subset^{\oplus} H^{\bullet}(X_{\overline{F}}^{\ast}, R \sigma^n_{\ast} k_{! \ast} \mathbb{Q}_{\ell}) \\
    &=H^{\bullet}(A^{\ast,n}_{\overline{F}}, k_{! \star} \mathbb{Q}_{\ell})\\
    &=IH^{\bullet}(A^{\ast, n}_{\overline{F}}, \mathbb{Q}_{\ell}),
\end{align}
and deduce that weight monodromy for the cohomology of  $j_{! \ast} \mathbb{V}$ follows from weight-monodromy for the intersection cohomology of $A^{\ast,n}$. Now let $\rho:Y \to A^{\ast, n}$ be a smooth projective morphism with $Y$ smooth projective.
\begin{Lem}[Corollary 1 of \cites{HansenBlog}]
The intersection cohomology complex of $A^{\ast,n}$ is a Galois-equivariant direct summand of
\begin{align}
    R \rho_{\ast} \mathbb{Q}_{\ell}.
\end{align}
\end{Lem}
\begin{proof}
We know that ${}^{\mathfrak{p}} \mathcal{H}^0(R \rho_{\ast} \mathbb{Q}_{\ell})$ is a Galois-equivariant direct summand of $R \rho_{\ast} \mathbb{Q}_{\ell}$ by Theorem 1.8 of \cites{MR3477743}. Let $a:U \to A^{\ast,n}$ be a dense open subset over which $\rho$ is smooth and such that $U$ is contained in $A^n$. Then there is a Galois-equivariant decomposition 
\begin{align}
    a_{! \ast} a^{\ast} {}^{\mathfrak{p}} \mathcal{H}^0(R \rho_{\ast} \mathbb{Q}_{\ell}) \subset^{\oplus}{}^{\mathfrak{p}} \mathcal{H}^0(R \rho_{\ast} \mathbb{Q}_{\ell})
\end{align}
by Lemma 2.2.8 of \cites{MR3477743}. We know that the intersection cohomology complex of $A^{\ast,n}$ is equal to $a_{! \ast} a^{\ast} \mathbb{Q}_{\ell}[d]$ and since intermediate extension is fully faithful it suffices to show that $\mathbb{Q}_{\ell}$ is a Galois-equivariant direct summand of $a^{\ast} {}^{\mathfrak{p}} \mathcal{H}^0(R \rho_{\ast} \mathbb{Q}_{\ell})$. But the latter is equal to
\begin{align}
    R^0 \tau_{\ast} \mathbb{Q}_{\ell},
\end{align}
where $\tau$ is the restriction of $\rho$ to $U$. The natural map
\begin{align}
    \mathbb{Q}_{\ell} \to R^0 \tau_{\ast} \mathbb{Q}_{\ell}
\end{align}
has a section coming from the trace map. 
\end{proof}
This means that weight-monodromy for $Y$ implies weight-monodromy for the intersection cohomology of $A^{\ast,n}$, proving the theorem.
\end{proof}

\DeclareRobustCommand{\VAN}[3]{#3}
\bibliography{references.bib}

\end{document}